\documentclass[a4paper,12pt]{elsarticle}
\usepackage{amsmath,amssymb,graphicx,amsthm}
\usepackage{algorithm}
\usepackage{algorithmicx}
\usepackage{algpseudocode}

\newtheorem{theorem}{Theorem}%[section]
\newtheorem*{theorem:repeatmain}{\tref{main}}

\newtheorem{lemma}[theorem]{Lemma}

\newtheorem{problem}[theorem]{Problem}
\newtheorem{prop}[theorem]{Proposition}

\newtheorem{definition}[theorem]{Definition}
\newtheorem{observation}[theorem]{Observation}

\newcommand\tref[1]{Theorem~\ref{thm:#1}}
\newcommand\cref[1]{Corollary~\ref{cor:#1}}

\newcommand\sref[1]{Section~\ref{sec:#1}}
\newcommand\pref[1]{Proposition~\ref{prop:#1}}

\newcommand\oref[1]{Observation~\ref{obs:#1}}

\newcommand\cC{{\mathcal C}}

\newcommand\cF{{\mathcal F}}
\newcommand\cG{{\mathcal G}}
\newcommand\cH{{\mathcal H}}

\newcommand\cQ{{\mathcal Q}}

\begin{document}
\begin{frontmatter}
\title{Finding a non-minority ball with majority answers}
\author[renyi]{D\'aniel Gerbner\fnref{dani}}
\author[renyi]{Bal\'azs Keszegh\fnref{keszegh,boly}}
\author[elte]{D\"om\"ot\"or P\'alv\"olgyi\fnref{dom,boly}}
\author[renyi,gac]{Bal\'azs Patk\'os\fnref{mate,boly}}
\author[renyi]{M\'at\'e Vizer\fnref{mate,mate2}}
\author[bme]{G\'abor Wiener\fnref{gabor,boly}}
\address[renyi]{MTA Alfr\'ed R{\'e}nyi Institute of Mathematics \\ Re\'altanoda u. 13-15 Budapest, 1053 Hungary\\
        {\tt email}: $<$gerbner.daniel,keszegh.balazs,patkos.balazs, vizer.mate$>$@renyi.mta.hu}
\address[elte]{E\"otv\"os Lor\'and University, Department of Computer Science. \\ P\'azm\'any P\'eter s\'et\'any 1/C Budapest, 1117 Hungary, \\ {\tt email:}
        {\it dom@cs.elte.hu}}
\address[gac]{MTA-ELTE Geometric and Algebraic Combinatorics Research Group. \\ P\'azm\'any P\'eter s\'et\'any 1/C Budapest, 1117 Hungary, \\ {\tt email:}
        {\it patkosb@cs.elte.hu}}
\address[bme]{ Department of Computer Science and Information Theory \\
Budapest University of Technology and Economics \\ 1117 Budapest, Magyar tud\'osok k\"or\'utja 2. \\ {\tt email:}
        {\it wiener@cs.bme.hu}}				
\fntext[dani]{Research supported by the Hungarian Scientific Research Fund (OTKA), under grant PD 109537.}
\fntext[keszegh]{Research supported by the Hungarian Scientific Research Fund (OTKA), under grant PD 108406.}
\fntext[dom]{Research supported by the Hungarian Scientific Research Fund (OTKA), under grant PD 104386.}
\fntext[mate]{Research supported by the Hungarian Scientific Research Fund (OTKA), under grant SNN 116095.}
\fntext[gabor]{Research supported by the Hungarian Scientific Research Fund (OTKA), under grant K 108947.}
\fntext[mate2]{Research supported by ERC Advanced Research Grant no 267165 (DISCONV).}
\fntext[boly]{Research supported by the J\'anos Bolyai Research Scholarship of the Hungarian Academy of Sciences.}

\begin{abstract}
Suppose we are given a set of $n$ balls $\{b_1,\ldots,b_n\}$ each colored either red or blue in some way unknown to us.
To find out some information about the colors, we can query any triple of balls $\{b_{i_1},b_{i_2},b_{i_3}\}$.
As an answer to such a query we obtain (the index of) a {\em majority ball}, that is, a ball whose color is the same as the color of another ball from the triple.
Our goal is to find a {\em non-minority ball}, that is, a ball whose color occurs at least $\frac n2$ times among the $n$ balls.
We show that the minimum number of queries needed to solve this problem is $\Theta(n)$  in the adaptive case and $\Theta(n^3)$ in the non-adaptive case. We also consider some related problems.
\end{abstract}
\begin{keyword} combinatorial search \sep majority \sep median \end{keyword}
\end{frontmatter}

\section{Introduction}
\label{sec:intro}
This paper deals with search problems where the input is a set of $n$ balls, each colored in some way unknown to us and we have to find a ball possesing a certain property (or show that such a ball does not exist) by asking certain queries.
Our goal is to determine the minimum number of queries needed in the worst case.
It is possible that the queries are all fixed beforehand (in which case we speak of {\em non-adaptive} search) or each query might depend on the answers to the earlier queries (in which case we speak of {\em adaptive} search).
We say that a ball $b$ is a {\em majority ball} of a set $A$ if there are more than $\frac{|A|}{2}$ balls in the set that have the same color as $b$.
Similarly, a ball $b$ is a {\em non-minority ball} of a set $A$ if there are at least $\frac{|A|}{2}$ balls in the set that have the same color as $b$.
Note that these two notions are different if and only if $n$ is even and $\frac{|A|}{2}$ balls are colored of the same color, in which case each of these balls is a non-minority ball and there are no majority balls (moreover, if there are only two colors, then the size of the other color class is also $\frac{|A|}{2}$, thus every ball is a non-minority ball).
For more than two colors it is possible that even non-minority balls do not exist.
A ball $b$ is said to be a {\em plurality ball} of a set  $A$ if the number of balls in the set with the same color as $b$ is greater than the number of balls with any of the other colors. In this paper we focus on the case of just two colors.

The most natural non-trivial question is the so-called {\em majority problem} which has attracted the attention of many researchers.
In this problem our goal is to find a majority ball (or show that none exists), such that the possible queries are pairs of balls $\{b_1,b_2\}$ and the answer tells us whether $b_1$ and $b_2$ have the same color or not.
In the adaptive model, Fisher and Salzberg \cite{FS} proved that $\lfloor 3n/2 \rfloor -2$ queries are necessary and sufficient for any number of colors, while Saks and Werman \cite{SW} showed that if the number of colors is known to be two, then the minimum number of queries is $n-b(n)$, where $b(n)$ is the number of 1's in the binary representation of $n$
(simplified proofs of the latter result were later found, see \cite{ARS,HKM,W}). In the non-adaptive model with two colors, it is easy to see that the minimum number of queries needed is $n-1$ if $n$ is even and $n-2$ if $n$ is odd.
% by Alonso, Reingold, and Schott \cite{ARS}, Wiener \cite{W} and Hayes, Kutin and van Melkebeek \cite{HKM}).

There are several variants of the majority problem \cite{A}. The plurality problem, where we have to find a plurality ball (or show that none exists) was considered, among others, in \cite{A,DJKKS,GKPP}.
Another possible direction is to use sets of size greater than two as queries \cite{MKW,DK}.

The main model studied in this paper is the following.
In the original comparison model the answer to the query $\{b_1,b_2\}$ can be interpreted as the answer to the question whether there is a majority ball in the subset $\{b_1,b_2\}$. If the answer is yes, then obviously both $b_1$ and $b_2$ are majority balls.
Therefore we obtain a generalization of the comparison model if for any query that is a subset of the balls the answer is either (the index of) a majority ball, or that there is no majority ball in the given subset (which cannot be the case if the size of the subset is odd and there are only two colors).
We study this model in case of two colors, and mostly when only queries of size three are allowed, although we also prove some results for greater query sizes.

Unfortunately, even asking all triples cannot guarantee that we can solve the majority problem for two colors.
Suppose we have an even number of balls that are partitioned into two sets of the same size, $X$ and $Y$, and suppose that the answer for any triple $T$ is a ball from $T\cap X$ if and only if $|T\cap X|\ge 2$. In this case we cannot decide whether all balls have the same color or all balls in $X$ are red, but all balls in $Y$ are blue.
In the former case all balls are majority balls, while in the latter there exists no majority ball.

Because of this, our aim will be to show a non-minority ball (which always exists if there are only two colors). Let us assume the balls are all red or blue and all queries are of size $q$.
We will denote the minimum number of queries needed (in the worst case) to determine a non-minority ball by $A_q(n)$ in the adaptive model and by $N_q(n)$ in the non-adaptive model.

At first sight the model we have just introduced seems to be rather artificial.
Let us, however, state a more natural problem that is equivalent to this model.

Suppose that our input is a binary sequence of length $n$, i.e., $n$ numbers %$a_1,a_2,\dots,a_n$
such that each is either 0 or 1.
Our task is to find a median element, such that the queries are {\em odd} subsets of the input elements and the answer is one of the median elements of the subset. Let us assume queries are of size $q$.
Denote the the minimum number of queries needed in the worst case to determine a median element by $A^{med}_q(n)$ in the adaptive model and by $N^{med}_q(n)$ in the non-adaptive model.

\begin{prop}
\label{prop:median}
$A_{2l+1}(n)=A^{med}_{2l+1}(n)$ and $N_{2l+1}(n)=N^{med}_{2l+1}(n)$.
\end{prop}
\begin{proof}
If we replace 0 and 1 by red and blue, then the median elements of any set are exactly the non-minority balls of the set.
\end{proof}

We obtain a natural generalization that also works for even sized subsets if the answer is the $t$th element for some fixed $t$. More precisely, for a query $Q$, the answer may be $a$ if and only if there exist $t-1$ elements $e\in Q \setminus \{a\}$, such that $e\geq a$ and $|Q|-t$ elements $e'\in Q \setminus \{a\}$, such that $e'\leq a$.
%(So in case of the median for an odd set of size $q$ we would have $t=\frac {q+1}2$.)
Note that in this model there might be more than one valid answer to a given query.  This can be outruled by assuming that all elements are different  (in which case we do not deal with just the numbers 0 and 1, obviously). This approach was proposed by G. O.H. Katona and studied by Johnson and M\'esz\'aros \cite{JM}.
They have shown that if all elements are different, then they can be almost completely sorted\footnote{
Note that the $t-1$ largest and the $q-t$ smallest elements cannot ever be differentiated with such questions, so we only want to determine these and sort the rest.}
using $O(n\log n)$ queries in the adaptive model and $O(n^{q-t+1})$ queries in the non-adaptive model and both results are sharp.
However, their algorithms fail if not all elements are different.
Our results imply that the same bound holds in the adaptive model with no restriction. However, the bound in the non-adaptive model cannot be extended to the general case.
%; consider the example ???
%While for adaptive sorting the proper variant of insertion sort makes only $O(n\log n)$ queries even if there are equal elements, we prove that in the non-adaptive model more queries are needed already to determine the median.
We discuss our related results in \sref{rem}.
%For example, from \pref{akkorisezazerosebb} and \tref{alfa} it follows that

%We denote the minimum number of queries needed (in the worst case) to determine the $t$th element in this model (with equal elements allowed) by $A^t_k(n)$ in the adaptive model and by $N^t_k(n)$ in the non-adaptive model.

To state our results concerning $A_3(n)$ and $N_3(n)$ we introduce the following notations.
We write $[n]=\{1,2,\dots,n\}$ for the set of the first $n$ positive integers and the set of balls is denoted by $B=[n]$.
For a set $S$, the set of its $k$-subsets will be denoted by $\binom{S}{k}$.
Let $\cQ \subseteq \binom{B}{3}$ be a query set.
Then for any ball $b\in B$ let $d_{\cQ}(b)=|\{Q\mid b\in Q\in \cQ\}|$ denote the \textit{degree} of $b$ in $\cQ$ and for any two balls $b_i,b_j\in B$ let $d_{\cQ}(b_i,b_j)=|\{Q\mid \{b_i,b_j\}\subset Q\in\cQ\}|$ denote the \textit{co-degree} of $b_i$ and $b_j$ in $\cQ$.
Furthermore, let us write $\delta(\cQ)=\min\{d_{\cQ}(b)\mid b\in B\}$ and $\delta_2(\cQ)=\min\{d_{\cQ}(b_i,b_j) \mid b_i,b_j\in B\}$.

Throughout the paper we use the following standard notation to compare the asymptotic behavior of two functions $f(n)$ and $g(n)$. We write $f(n)=o(g(n))$ if $\lim_n \frac{f(n)}{g(n)}=0$ holds. We write $f(n)=O(g(n))$ ($f(n)=\Omega(g(n))$) if there exists a positive constant $C$ such that $f(n)\le Cg(n)$ ($f(n)\ge Cg(n)$) holds for all values of $n$ and we write $f(n)=\Theta(g(n))$ if both $f(n)=O(g(n))$ and $f(n)=\Omega(g(n))$ hold. Sometimes, the function $f$ might have two variables $k$ and $n$. Then $f(k,n)=O_k(g(n))$ means that for every $k$ there exists a constant $C_k$ such that $f(k,n)\le C_kg(n)$ holds for all values of $n$. Finally, we write $f(n)=\tilde{O}(g(n))$ if there exist positive constants $C$ and $k$ such that $f(n)\le Cg(n)\log^kn$ holds.

\begin{theorem} \label{thm:adapt}
$A_3(n)=O(n)$.
%There exists a constant $C$ such that $A_3(n)\le Cn$ holds.
\end{theorem}

Before stating our results on $N_3(n)$ we state a theorem on the structure of query sets that do not determine non-minority balls.

\begin{theorem}
\label{thm:lower}

(i) If $|B|$ is odd and $\cQ\subseteq \binom{B}{3}$ is a set of non-adaptive queries with $\delta_2(\cQ)<\frac{n-2}{2}$, i.e.,
there is a pair of balls $x,y$ with $d_{\cQ}(x,y)<\frac{n-2}{2}$, then $\cQ$ cannot determine a non-minority ball.

(ii) For every $n$ there exists a non-adaptive query set $\cQ\subseteq \binom{B}{3}$ with $\delta_2(\cQ)=\lfloor n/2\rfloor +1$ that determines a non-minority ball.

(iii) If $|B|$ is even and $\cQ\subseteq \binom{B}{3}$ is a set of non-adaptive queries, such that there exist four balls $x,y,u,v$ with $d_{\cQ}(x,u)+d_{\cQ}(x,v)+d_{\cQ}(y,u)+d_{\cQ}(y,v)\le n/2-3$, then $\cQ$ cannot determine a non-minority ball.

(iv) Any non-adaptive query set $\cQ\subseteq \binom{B}{3}$ with $\delta_2(\cQ)>\frac{5n}{6}$ determines a non-minority ball.

(v) There exists a non-adaptive query set $\cQ\subseteq \binom{B}{3}$ with $\delta_2(\cQ)=\frac{5n}{6}-3$ that does not determine a non-minority ball.

\end{theorem}

With the help of \tref{lower} we will give bounds on $N_3(n)$.

\begin{theorem} \label{thm:degree}
(i) If $n$ is odd, then $N_3(n) \ge \frac{1}{2}\binom{n}{3}$.

\vspace{1mm}

(ii) If $n$ is even, then $ N_3(n) \ge \frac {1}{8}	\binom{n-2}{3}$.

\vspace{1mm}

(iii) $N_3(n)\le (\frac{5}{6}+o(1))\binom{n}{3}$.

\end{theorem}

\noindent
Concerning large query sizes we prove the following results:

\begin{theorem}\label{thm:parka}
We have $N_{2l+1}(n)=\Omega_l(n^3)$ and $A_{2l+1}(n)=O_l(n^2)$.
\end{theorem}

\begin{theorem}\label{thm:parosok}
We have $N_{2l}(n) \leq n(n-2l)$ and $A_{2l}(n) \leq n-2l+1$.
\end{theorem}

The rest of the paper is organized as follows: in \sref{algs} we examine some generalizations of the well-known median of medians algorithm of Blum, Floyd, Pratt, Rivest, and Tarjan \cite{bfprt}.
Using the facts shown in \sref{algs},  we prove \tref{adapt}, \tref{lower} and \tref{degree} in \sref{proofs}.
\sref{rem} contains the proof of \tref{parka} and \tref{parosok}, we introduce some new models and some open problems. We postpone some proofs and the analysis of a related model to the Appendix.

\section{Algorithms}
\label{sec:algs}

In this section we gather most of the building blocks of the algorithms we will use in Section 3 to prove our main results.
We start with an easy algorithm that finds two balls of different colors unless all balls have the same color.

\vskip 0.5truecm

\fbox{\textbf{Algorithm 2DB (Two Different Balls)}}

\vspace{3mm}

\textbf{Input}: a subset $S=\{b_1,b_2,\dots, b_m\}$ of the balls colored with two colors.

\vspace{3mm}

\textbf{Query}:  a triple $T=\{a,b,c\}$ of balls.

\vspace{3mm}

\textbf{Answer}: a majority ball in $T$ (that we will call \textit{answer ball}).

\vspace{3mm}

\textbf{Output}: two balls $b_i,b_j\in S$ with the property that either they are of different colors or all balls in $S$ have the same color.

\vspace{3mm}

\textbf{Description of Algorithm 2DB}

We start with an arbitrary query, then remove the answer ball, keep the other two elements and add a new element to obtain the second query. Then we repeat this procedure, always replacing the answer ball by a ball that has not appeared in any earlier query. After $m-2$ questions, we have removed $m-2$ balls, and thus we cannot continue this procedure and the algorithm outputs the remaining two balls. If these balls $b_i,b_j$ are of the same color, then it is easy to see that all balls in $S$ have the same color.

\vskip 0.1truecm

\textbf{Remark}: Algorithm 2DB is clearly adaptive, but if a non-adaptive query set $\cQ$ contains all queries from a subset $S$ of $B$, then Algorithm 2DB can be used. We will do so in \sref{proofs}.

From now on for a coloring $c$ and a ball $a$ we denote the color of $a$ by $c(a)$.

\begin{observation}
\label{obs:comparison} Suppose the balls are colored with 0 and 1, and we know that the color of $a_1$ is $0$ and the color of $a_2$ is $1$. If we query $\{a_1,b_1,b_2\}$ and $\{a_2,b_1,b_2\}$, then we can conclude one of the following.
%that $b_1$ and $b_2$ are of different colors, or we can compare the colors $c(b_1)$ and $c(b_2)$.
\begin{itemize}
\item $c(b_1)\ne c(b_2)$,
\item $c(b_1)\le c(b_2)$,
\item $c(b_1)\ge c(b_2)$.
\end{itemize}
\end{observation}

\begin{proof}
If the answers to the queries $\{a_1,b_1,b_2\}$ and $\{a_2,b_1,b_2\}$ are, respectively, $a_1$ and $a_2$, then $b_1$ and $b_2$ are of different colors. % as so are $a_1$ and $a_2$.
Otherwise, either the answer to $\{a_1,b_1,b_2\}$ is $b_i$, in which case $c(b_i)\le c(b_{3-i})$ holds,
or the answer to $\{a_2,b_1,b_2\}$ is $b_i$, in which case $c(b_i)\ge c(b_{3-i})$ holds.
\end{proof}

%\begin{varalgorithm}{MyBalls}
%\caption{Finding two balls of different colors if not all balls are of the same color}\label{2d}
%\begin{algorithmic}[1]
%\Procedure{Diff balls}{$b_1,b_2,\dots,b_n$}
% \State $Q_1 \gets \{b_1,b_2,b_3\}$
%\State $i \gets 1$
%    \While{$i \le n-3$ }
%		
%		\State $a_i \gets$ the answer to $Q_i$
		
%		\State $Q_{i+1} \gets Q_{i}\setminus \{a_{i}\} \cup \{b_{i+3}\}$
		
%		\State $i\gets i+$
%		\EndWhile
%		\State \textbf{Return} $Q_{n-2} \setminus \{a_{n-2}\}$
%	\EndProcedure
%\end{algorithmic}
%\end{varalgorithm}

\pref{median} showed the connection between the non-minority problem and finding a median element among $0-1$ entries. Algorithm 2DB and \oref{comparison} tell us that we can almost imitate comparison based algorithms to find the median with triple queries. Therefore it is natural to think that some of the comparison based median finding algorithms can be altered in a way that can be useful for our purposes. In the remainder of this section, we show two variants of the well-known median of medians (MoM) algorithm by Blum, Floyd, Pratt, Rivest, and Tarjan \cite{bfprt}. The first variant (MoM2) is a very natural generalization and is of independent interest: entries are not necessarily distinct integers, the answer to a query $\{a,b\}$ is either $a \le b$ or $b\le a$ and the Adversary has the right to answer any of $a\le b$ and $b \le a$ if $a$ and $b$ are equal. (Using the well-known Adversary method.) The second variant (MoM3) is much less natural, but is based on \oref{comparison}: entries are 0's and 1's and apart from the previous possibilities the Adversary has the right to answer $a\neq b$ if that is the case.

As in both models an element may appear more than once, therefore we need the following definitions in order to state our results.

Let $X$ be an $n$-element multiset of integers. We call an element $x$ of $X$ a {\em kth largest element} if there exists a partition of $X\setminus \{x\}=S\cup L$, such that $|L|=k-1$, $|S|=n-k$, such that each element in $L$ is at least $x$ and each element in $S$ is at most $x$. A \textit{median} is a ($\lfloor \frac{n}{2} \rfloor +1$)st largest element of an $n$-element set.

A \textit{decreasing enumeration} of an $n$ element multiset $X$ is a permutation $\sigma \in S_n$ with $x_{\sigma(1)} \ge x_{\sigma(2)} \ge \dots \ge x_{\sigma(n)}$.
Clearly, if all elements of $X$ are distinct, then $X$ has only one decreasing enumeration, while if in a multiset of integers the multiplicities are $k_1,k_2,\dots, k_l$, then the number of decreasing enumerations is $\prod^{l}_{i=1}k_i!$. We say that a multiset $X$ is completely sorted, if we fix one of its decreasing enumerations.

%We say that $X$ is \textit{sorted} by ITT MOST AZ VAN, HOGY X KUPACOKBA VAN OSZTVA, AMIK EGYENLOEK ES A KUPACOKNAK MEG VAN ADVA A RENDEZESE  the relations $R=\{i_1 \le j_1,i_2 \le j_2, \dots, i_m \le j_m\}$ if $x_{i_l} \le x_{j_l}$ holds for all $l=1,\dots, m$ and there exists a permutation $\sigma \in S_n$ such that $\sigma$ is a decreasing enumeration of any multiset $X'$ in which $x'_{i_l} \le x'_{j_l}$ holds for all $l=1,\dots, m$. For example, the relations $1 \le 2, 2\le 3, \dots, n-1 \le n, n\le 1$ sort $X$ as these relations tell us that all elements of $X$ are equal and thus any $\sigma \in S_n$ is a decreasing enumeration. Note that if $X$ is sorted, then for the appropriate decreasing enumeration $\sigma(k)$ is one of the $k$th largest element of $X$. The following observation shows that in the first comparison based model one can solve the problem of finding a $k$th largest element.

\begin{observation}
\label{obs:model1}
If for all $i,j \in [n]$ we know $x_i \le x_j$ or $x_i \ge x_j$, then we know a decreasing enumeration of $X$.
%$R$ sorts any $n$ element multiset $X$ that satisfies the relations of $R$. decreasing enumeration
\end{observation}

\begin{proof}
Consider the directed graph $D$ on vertex set $[n]$ with $(i,j) \in E(D)$ if and only if $x_i \le x_j $.
If $D$ is a DAG (directed acyclic graph), then $D$ is a transitive tournament and we are done.
If $D$ contains some directed cycles, then the integers in $X$ corresponding to all elements of such a directed cycle are equal.
Therefore we can contract all the vertices of the cycle into one vertex.
At some point no directed cycles will remain and we obtain the desired decreasing enumeration.
\end{proof}

The next algorithm shows that instead of asking all $\binom{n}{2}$ pairs, the problem of finding a $k$th largest element can be solved adaptively in linear time.

\vspace{5mm}

\fbox{\textbf{Algorithm MoM2}}

\vspace{3mm}

\textbf{Input}: an $n$ element multiset $X$ of integers and an integer $k$ with $1 \le k\le n$.

\vspace{3mm}

\textbf{Goal}: find one of the $k$th largest elements of the input multiset.

\vspace{3mm}

\textbf{Query}:  a pair $\{a,b\}$.

\vspace{3mm}

\textbf{Answer}: either $a \le b$ or $b \le a$ (in case $a=b$, both answers are possible).

\vspace{3mm}

\textbf{Output of Algorithm MoM2}: ($x$, $S$, $L$), a partition of the input multiset, where:

\vspace{1mm}

\ $\circ$ $x$ is one of the $k$th largest elements, %out of $n$ (not necessarily different) elements,

\vspace{1mm}

\ $\circ$  $L$ is a set of $k-1$ elements that are at least $x$, and

\vspace{1mm}

\ $\circ$  $S$ is a set of $n-k$ elements that are at most $x$.

\vspace{3mm}

\textbf{Description of Algorithm MoM2}

\vspace{2mm}

The Algorithm MoM2 consists of 4 phases. For each phase we will write in teletype style what the algorithm does, and our analysis will be in normal typestyle.

\vspace{1mm}

In the description of the algorithm we will introduce two sets: $S'$ and $L'$, change them dynamically during the phases and finally use them to define $S$ and $L$ of the output. (For the sake of simplicity we will use the term set instead of multiset in the description of the algorithm.)

\vspace{1mm}

To count the queries used by the recursive calls, we denote by $f(n)$ the worst case running time of \textit{Algorithm MoM2} on $n$ elements for any $k$.

\vspace{3mm}

\textbf{Phase 1:} \textsf{We divide the elements of the $n$-element input set into groups of five except at most four elements and sort each group.}

\vspace{2mm}

\oref{model1} shows that this can be done with 10 comparisons in each group, but a simple case analysis shows that adaptively 7 queries are enough.
Therefore this phase requires $7\cdot \lfloor 0.2n\rfloor$ queries.
Take a median from each group to form $M$ of size $\lfloor 0.2n \rfloor$.

\vspace{3mm}

\textbf{Phase 2:} \textsf{By recursion we can find a median $p$ of $M$, called the \textit{pivot}, and a partition of the rest of the elements of $M$ into two almost equal subsets, $S'$ and $L'$, such that each element of $S'$ is at most as large as the pivot and each element of $L'$ is at least as large as the pivot.}

\textsf{Then we put into $S'$ the elements that were smaller than or equal to some $s\in S'\cup\{p\}$ in their (completely sorted) group and we put into $L'$ the elements which were larger than or equal to $\ell\in L'\cup\{p\}$ in their (completely sorted) group.}

\vspace{2mm}

We use $f(\lfloor 0.2n \rfloor)$ queries during this phase.
Note also that at the end of this phase we have $$|L'|,|S'|\ge \lfloor 0.3(n-4)\rfloor -1.$$

\vspace{3mm}

\textbf{Phase 3:} \textsf{ We compare each element $e$ in the complement of $S' \cup L' \cup \{p\}$ to the pivot and if the answer is $e \leq p$, we put $e$ in $S'$, if the answer is $p \leq e$, we put $e$ in $L'$.}

\vspace{2mm}

This phase requires at most $\lceil 0.4(n-4) \rceil+6$ queries. At the end of this phase we have $$\lfloor 0.3(n-4)\rfloor -1 \le|S'|,|L'|\le \lceil 0.7(n+4)\rceil +2.$$

\vspace{3mm}

\textbf{Phase 4:}

\vspace{2mm}

\textbf{Case 1:} \textsf{If $|L'| = k-1$ (i.e., the pivot is a $k$th largest element), then the output of the algorithm is $(p,S',L')$.}

\vspace{2mm}

\textbf{Case 2:} \textsf{If $|L'| > k-1$, then by a recursive call on $L'$, whose output is $(x,S'',L'')$, we find $x$, a $k$th element of $L'$
and a partition of $L' \setminus \{x\}=S'' \cup L''$ such that $|L''|=k-1$ and $|S''|=|L'|-k$.
In this case the output of the algorithm is $(x, S'\cup \{p\} \cup S'', L'')$.}

\vspace{2mm}

\textbf{Case 3:} \textsf{If $|L'| < k-1$, then by a recursive call on $S'$ we find $x$, a $(k-|L'|-1)$th element of $S'$ and a partition of $S' \setminus \{x\}=S'' \cup L''$ such that $|L''|=(k-|L'|-1)-1$.
In this case the output of the algorithm is $(x, S'', L'\cup \{p\} \cup L'')$.}

\vspace{4mm}

In each cases of Phase 4 we find the respective element, using at most $$f(\max(|S|;|L|))\le f(\lceil 0.7(n+4) \rceil+2) $$ queries. During the algorithm altogether we have used
$$f(n) \le 7\cdot \lfloor 0.2n\rfloor + f(\lceil 0.2n \rceil)+ \lceil 0.4(n-4) \rceil + 6 + f(\lceil 0.7(n+4) \rceil+2)$$

\noindent
queries. Now by induction $f(n)\le 18n+7$.
It is worth mentioning, that this is somewhat better than the well-known $22n$ bound for all different numbers using the original algorithm.
(Note that the best bound for finding the median is between $2n$ and $3n$, see \cite{DZ}.)

\vskip 0.5truecm

Let us now consider the second model, that is, where the input is restricted to binary sequences, but the Adversary can answer $a\le b$, $b\le a$ or $a\ne b$. %Although, this time the Adversary must follow some rules (for example there cannot exist any odd cycles in the graph where an edge is defined for every answer $a\neq b$, but there are some other restrictions), EZ MAR KORABBINAL IS VOLT, NINCS UJDONSAG
In this model for any $k$ there is a strategy of the Adversary, such that even asking all possible queries we can not find any of the $k$th largest elements. Indeed, let us partition $X$ into two sets of equal size $X_1$ and $X_2$. If the Adversary answers $x_1 \neq x_2$ whenever $x_1\in X_1$ and $x_2 \in X_2$, then all we know is that there are the same number of 0's and 1's in $X$ and the two classes are $X_1$ and $X_2$, but we cannot tell which one is which, therefore we cannot solve the problem.

On the other hand, using Algorithm 2DB, for balls we can suppose that a 0 and a 1 is in $X=\{x_1,x_2,\dots,x_n\}$, provided not all elements of $X$ are the same. Still, a similar strategy of the Adversary shows that even after querying all pairs, we cannot sort all the elements. However, we can show a $k$th largest element for any $k$.

\begin{observation}
\label{obs:model2} If the Adversary answers all possible queries, then %either we find a contradiction in his answers, or EDDIG SEM VOLT ERROL SZO, NE KEZDJUK EL MOST

(i) we can partition $X=O\cup R$ such that $O$ is completely sorted and $|R|=2r$ contains exactly $r$ many 0's and 1's.
In particular, if $n$ is odd then we can find a median element.

(ii) We can determine a $k$th largest element provided $x_1=0,x_2=1$ hold.
\end{observation}

\begin{proof}
We can check if there exists an $n$ element multiset satisfying all answers of the Adversary. If not, then we are done, as we revealed a contradiction in his answers. However if there is an $n$ element multiset $X$ satisfying Adversary's answers, then let us remove a subset $R \subseteq X$ in the following way: let $R_0=\emptyset$ and as long as there exists a pair $x_i,x_j \in X\setminus R_t$ with the answer $x_i\neq x_j$, let us put $R_{t+1}=R_t\cup\{x_i,x_j\}$. If there is no such pair, then we stop and write $R=R_t$. As removed elements come in pairs, we obtain that $|R|=2r$ for some $r$ and $R$ contains exactly $r$ 0's and 1's. If $k\le r$, then a $k$th largest element must be 1, thus we can output $x_2=1$. Similarly, if $k\ge n-r+1$, then a $k$th largest element must be 0 and we can output $x_1=0$. Finally, if $r<k<n-r$ holds, then by \oref{model1} we can sort $O=X\setminus R$ and output a $(k-r)$th element of $X\setminus R$. To see the second part of (i), note that if $n$ is odd, then $r<\frac{n+1}2<n-r$ always holds. %$O$ cannot be empty as $|R|$ is even.
\end{proof}

\vspace{5mm}

\fbox{\textbf{Algorithm MoM3}}

\vspace{3mm}

\textbf{Input}: an $n$-element multiset $X=\{x_1, x_2,x_3\dots,x_n\}$ containing 0's and 1's (and we will

use
$\textbf{0}$ instead of $x_1$ and $\textbf{1}$ instead of $x_2$) and an integer $k$ with $1\le k \le n$.

\vspace{3mm}

\textbf{Goal}: find one of the $k$th largest elements of the input multiset.

\vspace{3mm}

\textbf{Query}:  a pair $\{a,b\}$.

\vspace{3mm}

\textbf{Answer}: either $a \le b$, $b \le a$ or $a \neq b$.

\vspace{4mm}

\textbf{Output of Algorithm MoM3}: ($x,R,S,L$), a partition of the input multiset where:

\vspace{1mm}

\ $\circ$ $x$ is one of the $k$th largest elements of $X$.

\vspace{1mm}

\ $\circ$ $|R|=2r$ such that $r$ of them are 0's and $r$ of them are 1's.
Moreover, the output also contains a bijection among the 0's and 1's of $R$.
Using this bijection, we will talk about the {\em pair} of an element of $R$.

\vspace{1mm}

\ $\circ$ $S \cup L$, a partition of $X\setminus (\{x\} \cup R)$ such that each element in $S$ is at most and each element of $L$ is at least as large as $x$ with $\frac{|R|}{2}+|L| = k-1$ if $\frac{|R|}{2} \le \min \{n-k,k-1\}$.

If $\frac{|R|}{2} > k-1$ we output $(\textbf{1}, R, X \setminus (\textbf{1} \cup R), \emptyset)$ and  we output $(\textbf{0}, R, \emptyset, X \setminus (\textbf{0} \cup R))$ if $\frac{|R|}{2} > n-k$. (Note that $\frac{|R|}{2} > \max \{n-k, k-1\}$ cannot happen.)
%The known 0 is in $S$, the known 1 is in $L$. - HA VISSZARAKJU, $\textbf{0}$ KELL

\vspace{3mm}

\textbf{Description of Algorithm MoM3}

\vspace{3mm}

Algorithm MoM3 consists of 5 phases. For each phase we will write in teletype style what the algorithm does, and our analysis will be in normal typestyle.

\vspace{2mm}

In the description of the algorithm we will introduce three sets: $R'$, $S'$ and $L'$, change them dynamically during the phases and finally use them to define $R$, $S$ and $L$ of the output. We make sure that at any moment of the algorithm $R'$ consists of pairs of 0's and 1's. (Again, for the sake of simplicity we will use the term set instead of multiset.)

\vspace{2mm}

To count the queries used by the recursive calls, we denote by $g(n)$ the worst running time of \textit{Algorithm MoM3} on $n$ elements for any $k$. For the sake of simplicity during the computation we omit the additive constants, floor and ceiling signs.

\vspace{4mm}

\textbf{Phase 1:}  \textsf{We divide the elements of $X\setminus \{\mathbf{0,1}\}$ into groups of five (with the exception of at most four elements) and execute all possible queries within all groups. Applying \oref{model2} (i) we obtain a median element in each of the groups.}

\vspace{2mm}

Let $M$ denote the set of medians, then $|M|=0.2n$. For any $m\in M$, let $G_m$ denote its five element group, and let $O_m$ and $R_m$ be the partition of $G_m$ we obtain by \oref{model2}. We put $R'=\cup_{m\in M}R_m$ and note that $O_m$ has size 1, 3, or 5 for all $m\in M$.

\vspace{2mm}

This phase requires ${5\choose 2}\cdot  0.2n $ queries.

\vspace{3mm}

\textbf{Phase 2:}  \textsf{By a recursive call on $M\cup \{\mathbf{0}, \mathbf{1}\}$, we find a median $p$ of $M$, called the \textit{pivot}, a set of pairs $R_1$ (we do {\em not} put these elements into $R'$), and a partition of the rest of the elements into two subsets, $S'$ and $L'$ with $0 \le |S'|-|L'|\le 1$, such that each element of $S'$ is at most as large as the pivot and each element of $L'$ is at least as large as the pivot.
For all $s\in S'$ we put all elements of the sorted $O_s$ not larger than $s$ into $S'$ and for all elements $\ell\in L'$ we put all elements of the sorted $O_{\ell}$ not smaller than $\ell$ into $L'$.}

\vspace{2mm}

Note that both $L'$ and $S'$ can contain $0's$ and $1's$.
%We also put the elements of $R'$ into $R$.
%As the pivot was a median of the medians, we have $|S'|, |L'|\le \frac n5$ and also $|S\cup R|\ge $

\vspace{2mm}

This phase requires $g(0.2n)$ queries.

\vspace{3mm}

\textbf{Phase 3:}  \textsf{For every pair $a,b\in R_1 \subset M$ we query all pairs $\{c,d\}$ with $c \in G_a,d\in G_b$.
Using the fact that one of $a$ and $b$ is 0, the other one is 1, we can apply \oref{model2} (i) to $G_a \cup G_b$ and obtain the partition $G_a\cup G_b=O_{a,b}\cup R_{a,b}$ with $a,b \in R_{a,b}$.}

\textsf{$\bullet$ If $|R_{a,b}| \ge 6$, we put its elements into $R'$,}

\textsf{$\bullet$ if not, then %$|O_{a,b}|=8$. VAGY 6, UYGYE? As $a\neq b$ are median elements of $G_a$ and $G_b$, we know that in this case $O_{a,b}=G_a\cup G_b \setminus \{a,b\}$ contains at least two 0's and at least two 1's.
we can still deduce that some elements of $G_a$ and $G_b$ are 0 and 1 using that $a\neq b$ are median elements of $G_a$ and $G_b$.
Therefore we can pair the first one and the last one (if $|R_{a,b}|=4$) or the first two and the last two (if $|R_{a,b}|=2$) elements of the order of $O_{a,b}$ and put them into $R'$ along with $a$ and $b$.}

\vspace{2mm}

Note that at this moment by Phase 2 for every $m \in M\setminus R_1$ we have at least 3 elements of $G_m$ in $S' \cup L'\cup R'$ and by Phase 3 for every pair $a,b \in R_1$ we have at least 6 elements of $G_a\cup G_b$ in $R'$. Thus $|S'\cup L'\cup R'|\ge 0.6n$ holds. Observe that at this moment we have $|S'|, |L'|\le 0.3n$ as half of the groups $G_x$, $x \in M\setminus R_1$, contributed at most 3 elements to $S'$ and the other half of such groups contributed at most 3 elements to $L'$, while groups $G_x$ with $x \in R_1$ contributed only to $R'$. Let $Z=X\setminus (S'\cup L'\cup R')$.

\vspace{2mm}

This phase requires $5^2\cdot \frac {|R_1|}2\le  2.5n $ queries.

\vspace{3mm}

\textbf{Phase 4:}  \textsf{For every element $x \in Z$ we query the pair $\{x,p\}$.
If the answer is $x\le p$, then we put $x$ into $S'$, if the answer is $x\ge p$, then we put $x$ into $L'$.}

 %and if the answer is $x\le$ pivot, then we put $x$ into $S$, if the answer is $x\ge$ pivot, then we put $x$ into $L$, while if the answer is $x\neq$ pivot, we put $x$ into a subset called $M$.

\vspace{2mm}

Observe that new elements to $S'$ and $L'$ came from $Z$, thus at this moment their size is not more than $0.7n$. Let $T$ be the set of those elements, when the answer is $x\neq p$. Note that all elements of $T$ have the same color (the opposite of that of $p$) and as $T \subseteq Z$, we have $|T| \le 0.4n$.

\vspace{2mm}

By the above observation, we use at most $0.4n$ queries.

\vspace{3mm}

\textbf{Phase 5:}  \textsf{In this phase, we compare every element of $S' \cup L'$ to one element of $T$. We proceed in the following way:}

\vspace{1mm}

\textsf{$\bullet$ if the answer to query $\{x,t\}$ with $x\in S'\cup L'$, $t \in T$ is $x\neq t$, then we put $x$ and $t$ as a pair to $R'$ and the remaining elements of $S'\cup L'$ should be compared to a remaining element of $T$,}

\textsf{$\bullet$ if the answer is $x\le t$ or $t \le x$, then we just move on to the next element of $S'\cup L'$.}

\vspace{2mm}

We use at most $|S'|+|L'|\le n$ queries.

\vspace{5mm}

\noindent
The following can occur during Phase 5:

\vspace{2mm}

\textsf{If in any of the following cases $\frac{|R'|}{2} > k-1$ happens, we output $(\mathbf{1}, R', L'\cup S' \cup \{\mathbf{0}\}, \emptyset )$ and if $\frac{|R'|}{2} > n-k$ happens, we output $(\mathbf{0}, R', \emptyset , L'\cup S'\cup \{\mathbf{1}\})$.}

\vspace{3mm}

\textbf{Case 1:} $T$ becomes empty as all its elements are moved to $R'$.

\vspace{2mm}

In this case $X$ is partitioned into $p$, $S'$, $L'$, and $R'$.
By the above case we are not done if $\frac{|R'|}{2} < k-1 < n-\frac {|R'|}2-1$, but then observe that the $k$th largest element out of the $n$ original elements is the same as the $(k-\frac {|R'|}2)$th element from $S'\cup L'\cup \{p\}$. We know that $s \le p \le \ell$ for all $s \in S', \ell \in L'$.

\vspace{2mm}

\textsf{Therefore we can make a recursive call to either $L' \cup \{\mathbf{0}, \mathbf{1}\}$ with $k'=k-\frac {|R'|}2$ or $S'\cup \{\mathbf{0}, \mathbf{1}\}$ with $k'=k-\frac {|R'|}2- |L'|-1$ depending on whether $k-\frac {|R'|}2\le |L'|$ or $k-\frac {|R'|}2 > |L'|+1$ (if $k-\frac {|R'|}2=|L'|+1$, then $p$ is a $k$th largest element of $X$ and we can output $(p,R',S'\cup \textbf{0}, L' \cup \textbf{1})$).}

\textsf{If the recursive call on $L'\cup\{\mathbf{0},\mathbf{1}\}$ outputs $(\ell, R'', S'', L'')$, then our final output is $(\ell, R' \cup R'', S' \cup S'', L'')$. The case when we make the recursive call to $S'\cup \{\mathbf{0},\mathbf{1}\}$ is similar.}

\vspace{2mm}

As $|S'|,|L'|<0.7n$, this requires at most $g(0.7n)$ queries.

\vspace{4mm}

\textbf{Case 2:} We obtain an answer $t \leq s$  for some $t \in T$ and $s \in S'$.

\vspace{2mm}

In this case the value of the pivot is $1$, all elements of $T$ are $0$ and all elements of $L'$ are $1$.
Indeed, supposing that the value of the pivot is 0, then as for all $t\in T$ we have $t\neq p$, we obtain $1=t\le s \le p=0$, a contradiction. As the value of the pivot is 1, so are the values of all $\ell \in L'$. Thus

\vspace{2mm}

\textsf{$\bullet$ if $k \le \frac{|R'|}{2}+|L'|+1$, then we can output $(\mathbf{1},R',T\cup S'\cup L^{--} \cup \mathbf{0}, L^-)$, where $L^-$ is an arbitrary subset of $L'$ of size $k-\frac{|R'|}{2}-1$ and $L^{--}= L' \setminus L^- $,}

\vspace{1mm}

\textsf{$\bullet$ if $k > n-\frac{|R'|}{2}-|T|$, we can output $(\mathbf{0}, R',S^- \cup T,S^{--}\cup L' \cup \mathbf{1})$,  where $S^-$ is a subset of $S'$ of size $n- \frac{|R'|}{2}-|T|-k$ and $S^{--}=S' \setminus S^-$, since half of the elements in $R'$ and all elements in $T$ are 0,}

\vspace{1mm}

\textsf{$\bullet$ if $\frac{|R'|}{2}+|L'|+1<k\le  n-\frac{|R'|}{2}-|T|$, then a $(k-(\frac{|R'|}{2}+|L'|+1))$st element of $S'$ is a $k$th element of $X$. We make a recursive call to $S'\cup \{\textbf{0},\textbf{1}\}$ with $k'=k-(\frac{|R'|}{2}+|L'|+1)$.  If the output of the recursive call is $(x,R'',S'',L'')$, then our final output is $(x,R'\cup R'', T\cup S'', L'\cup L'' \cup \{p\})$.}

\vspace{2mm}

This case uses at most $g(|S'|+2)\le g(0.7n)$ queries.

\vspace{3mm}

\textbf{Case 3:} We obtain an answer $\ell \leq t$ for some $t \in T$ and $\ell \in L'$.

\vspace{2mm}

In this case the value of the pivot is $0$, all elements of $T$ are $1$ and all elements of $S'$ are $0$ and \texttt{we proceed analogously to the previous case.}

\vspace{4mm}

\textbf{Case 4:} We have $s \leq t$ and $t \leq \ell$ for all $t \in T$, $s \in S'$, $\ell \in L'$.

\vspace{2mm}

In this case all elements of $S'$ have value $0$ and all elements of $L'$ have value $1$. Indeed, we have $0\le s \le p \le \ell\le 1$, $0\le s \le t \le \ell\le 1$ and $t\neq p$.

\vspace{2mm}

\textsf{We put one element $t\in T$ and $p$ to $R'$ as a pair and}

\vspace{2mm}

\textsf{$\bullet$ if $k-1 \le \frac{|R'|}{2}+|L'|$, then we can output $(\mathbf{1}, R', S'\cup (T \setminus \{t\})\cup L^{--}, L^-)$, where $L^-$ is a subset of $L'$ of size $k-\frac{|R'|}{2}-1$ and $L^{--}= L' \setminus L^-$,}

\textsf{$\bullet$ if $k\ge n-\frac{|R'|}{2}-|S'|$, then we output $(\mathbf{0}, R', S^- ,S^{--}\cup (T \setminus \{t\})\cup L')$, where $S^-$ is a subset of $S'$ of size $n- \frac{|R'|}{2}-k$ and $S^{--} = S' \setminus S^-$,}

\textsf{$\bullet$ if $\frac{|R'|}{2}+|L'|+1< k < n-\frac{|R'|}{2}-|S'| $, then we output $(t', R', S'\cup T^{--}, T^{-}\cup L')$, where $t'\in T\setminus\{t\}$, $T^-\subseteq T \setminus \{t,t'\}$ with $|T^{-}|=k-1-\frac{|R'|}{2}-|L'|$ and $T^{--}=(T \setminus \{t,t'\})\setminus T^-$.}

\vspace{3mm}

In all cases of the final case analysis we used at most $ g( 0.7n )$ queries. Therefore we obtain that the running time $g(n)$ satisfies

$$g(n) \le 2n + g(0.2n)+ 2.5n + 0.4n + n+ g(0.7n).$$
Solving this we obtain a linear bound $g(n)\le 59n+O(1)$.

\section{Proofs of the main theorems}
\label{sec:proofs}

In this section we prove Theorems 2, 3 and 4, and to make the presentation more followable we restate them before their proof.

First we put together the pieces from \sref{algs} to obtain a proof of \tref{adapt}.

\vspace{3mm}

\noindent
\textbf{Theorem 2.}
\textit{$A_3(n)=O(n)$.}

\vspace{1mm}

\begin{proof}[Proof of \tref{adapt}]
Let us start by executing Algorithm 2DB.
If the two balls of the output have the same color, then no matter which ball we output, it will be a non-minority ball.
Thus, from now on we assume that the two remaining balls are of different colors.
We call the color of one of them 0, the other 1, and denote the respective balls by $\textbf{0}$ and $\textbf{1}$.
To every ball we assign the number of its color, e.g., for two balls $a\le b$ means that if $a$ has color 1, then so does $b$.
By \oref{comparison}, we know that after obtaining the answers to the queries $\{\textbf{0},a,b\}$ and $\{a,b,\textbf{1}\}$ we know if $a\le b$, $b \le a$ or $a \neq b$ hold.

Therefore, we can run in linear time the queries that correspond to the queries of Algorithm MoM3 to find a median, which is, by Proposition \ref{prop:median}, a non-minority ball.
\end{proof}

Now we turn our attention to non-adaptive problems. We start with a definition and a simple observation that we will use in many of our proofs.

\vspace{2mm}

\begin{definition}
Let $\cQ$ be a non-adaptive query set and $(x_1,x_2,\dots,x_s)$ ($x_i\in Q_i$) a possible sequence of answers.
We say that $c:B\rightarrow \{0,1\}$ is a \textit{legal coloring} of the ball set $B$ if for every $1\le i \le s$ $x_i$ is a majority ball in $Q_i$.
The \textit{minority set} $M_c$ of a coloring $c:B\rightarrow \{0,1\}$ is the set of all balls that are not non-minority balls.
%$c^{-1}(0)$ if $|c^{-1}(0)|<|c^{-1}(1)|$, $c^{-1}(1)$ if $|c^{-1}(0)|>|c^{-1}(1)|$ and $\emptyset$ if $|c^{-1}(0)|=|c^{-1}(1)|$.
\end{definition}

\begin{observation}
\label{obs:cover} A non-adaptive query set $\cQ=\{Q_1,Q_2,\dots,Q_s\}$ does not determine a non-minority ball if and only if there exists a sequence $(x_1,x_2,\dots,x_s)$ ($x_i\in Q_i$) of answers for which the minority sets of all legal colorings cover the ball set, i.e., $B=\cup_{c\in\cC} M_c$ where $\cC$ denotes the set of all legal colorings.
\qed
\end{observation}
%\begin{proof}
%A query set $\cQ$ does not determine a non-minority ball if and only if there exists a set of answers such that for every ball $b$ there exists a legal coloring $c$ for which $b$ is a minority ball, i.e., $b\in M_c$.
%\end{proof}

Using the above simple observation, we can prove \tref{lower} that we restate here.

\vspace{3mm}

\noindent
\textbf{Theorem 3.}

\vspace{2mm}

\textit{(i) \ If $|B|$ is odd and $\cQ\subseteq \binom{B}{3}$ is a set of non-adaptive queries with $\delta_2(\cQ)<\frac{n-2}{2}$, i.e.,
there is a pair of balls $x,y$ with $d_{\cQ}(x,y)<\frac{n-2}{2}$, then $\cQ$ cannot determine a non-minority ball.}

\vspace{1mm}

\textit{(ii) \ For every $n$ there exists a non-adaptive query set $\cQ\subseteq \binom{B}{3}$ with $\delta_2(\cQ)=\lfloor n/2\rfloor +1$ that determines a non-minority ball.}

\vspace{1mm}

\textit{(iii) \ If $|B|$ is even and $\cQ\subseteq \binom{B}{3}$ is a set of non-adaptive queries, such that there exist four balls $x,y,u,v$ with $d_{\cQ}(x,u)+d_{\cQ}(x,v)+d_{\cQ}(y,u)+d_{\cQ}(y,v)\le n/2-3$, then $\cQ$ cannot determine a non-minority ball.}

\vspace{1mm}

\textit{(iv) \ Any non-adaptive query set $\cQ\subseteq \binom{B}{3}$ with $\delta_2(\cQ)>\frac{5n}{6}$ determines a non-minority ball.}

\vspace{1mm}

\textit{(v) There exists a non-adaptive query set $\cQ\subseteq \binom{B}{3}$ with $\delta_2(\cQ)=\frac{5n}{6}-3$ that does not determine a non-minority ball.}

\vspace{3mm}

\begin{proof}[Proof of \tref{lower}]
First we prove (i).
Let $|B|=2k+1$ and assume that for a query set $\cQ$ and a pair of balls $x,y \in B$ we have $d_{\cQ}(x,y)< k$.
We have to show that we cannot determine a non-minority ball.
Let us partition $B\setminus \{x,y\}$ into two sets $B_1$ and $B_2$ such that $|B_1|=k$, $B_2=k-1$, and $\{b\in B: \{x,y,b\}\in\cQ\} \subseteq B_2$. To prove (i) we will show how to answer queries of $\cQ$ such that the conditions of \oref{cover} are met. Let $Q\in\cQ$ be an arbitrary query.

\begin{enumerate}
\item
if $|Q\cap B_i|\ge 2$ for some $i$, then the answer to $Q$ is a ball from $Q\cap B_i$,
\item
if $|Q\cap B_i|=1$ for $i=1,2$, then we answer the ball from $Q\cap \{x,y\}$,
\item
if $\{x,y\}\subset Q$, then by the assumption on the partition we know that the third ball in $Q$ belongs to $B_2$, and the answer is this third ball.
\end{enumerate}

Note that the above answers are all possible if we assume that balls in $B_1$ are blue and balls in $B_2$ are red. Furthermore, (3)  assures that at least one of $x$ and $y$ is also red. Thus the three different colorings  $c_1,c_2,c_3$ for which $B_1$ is blue, $B_2$ is red and at least one of $x,y$ is red, with respective minority sets $M_1=B_1, M_2=B_2 \cup \{x\}, M_3=B_2\cup \{y\}$, are all legal with respect to the above answers. Therefore by \oref{cover} $\cQ$ does not determine a non-minority ball.

To prove (ii) we construct a query set $\cQ\subseteq \binom{B}{3}$ with $\delta_2(\cQ)=\lfloor n/2\rfloor +1$ that determines a non-minority ball.
Let $S\subseteq B$ be a subset of the balls with $|S|=\lfloor n/2\rfloor+1$ and let $\cQ=\binom{B}{3}\setminus \binom{B\setminus S}{3}$. In proving that $\cQ$ does indeed determine a non-minority ball we will apply our results concerning the adaptive algorithms from  \sref{algs}. %Obviously it is not allowed in general. However, here we make sure that all the queries that the adaptive algorithm asks appear in $\cQ$.

We start by executing Algorithm 2DB on $S$ (we can do that as $\binom{S}{3} \subseteq \cQ$).
We obtain two balls $a$ and $b$ that have different colors unless all balls in $S$ have the same color. If $a$ and $b$ are of the same color, then that color is the majority color (and hence non-minority), and if some $x$ is colored with the other color, $x$ cannot be in $S$, therefore the answer to the query $\{a,x,b\}$ cannot be $x$.
We look at the queries of the form $\{a,x,b\}$ for all $x \not\in S$, and define $R$ to be the set of those balls $x\in B\setminus S$, for which the answer to the query $\{a,x,b\}$ is $a$ or $b$.
We will make sure that our final output will be a ball in $B\setminus R$. This guarantees that if the colors of $a$ and $b$ are the same, then we will output a non-minority ball.
Therefore, we can assume that the color of $a$ is 0 and the color of $b$ is 1.
If the answer to the query $\{a,x,b\}$ is $a$, then the color of $x$ is 0, if the answer is $b$, then the color of $x$ is 1. Let $n_1$ denote the number of balls $x$ of the latter type. We remove the balls in $R$, and then the median is the $k$th largest among the remaining balls for $k=\lceil n/2 \rceil -n_1$ (this is always positive as $|S|=\lfloor n/2\rfloor +1$), and we can find it using Algorithm MoM3.
All these queries are in $\cQ$, as they contain $a\in S$ or $b\in S$.
%If our assumption was wrong and the two balls 0 and 1 are of the same color, then by \oref{comparison} every ball in $S$ is of that color, hence that is majority color on the whole input. The ball $a$ that we claim to be the median was the answer to the query $0a1$, hence it is of that color too.

For (iii), we have to show that if $n=|B|$ is even and $\cQ\subseteq \binom{B}{3}$ is a set of non-adaptive queries such that there exist four balls $x,y,u,v$ with $d_{\cQ}(x,u)+d_{\cQ}(x,v)+d_{\cQ}(y,u)+d_{\cQ}(y,v)\le n/2-3$, then $\cQ$ cannot determine a non-minority ball.
Take the set of balls that are in a query with one of $\{x,y\}$ and one of $\{u,v\}$, e.g., $\{bxu\}$,
and add some further balls to them, if necessary, to form a set $B_1$ of size $n/2 -3$.
Let $B_2=B\setminus (B_1\cup \{x,y,u,v\})$ be the set of the remaining $n/2-1$ balls.
We answer the queries such that all colorings are valid for which $B_1$, $B_2$, $\{x,y\}$ and $\{u,v\}$ are all monochromatic sets, $B_1$ and $B_2$ are colored differently and either $x$ and $y$, or $u$ and $v$ (possibly all four) have the same color as the balls in $B_1$.
\begin{itemize}
\item
If a query $Q$ meets one of the four sets above in at least two balls, then the answer is one of those balls,
\item
the answer to a query $\{\xi, \nu, b\}$ with $\xi \in \{x,y\}$, $\nu \in \{u,v\}$, $b\in B_1$ is $b$,
\item
the answer to a query $\{b_1,b_2,z\}$ with $b_1 \in B_1,b_2 \in B_2, z \in\{x,y,u,v\}$ is $z$.
\end{itemize}
By the definition of the partition $B_1,B_2$, there are no other possible queries, and one can easily check that the sets $B_2, B_1\cup \{x,u\}$ and $B_1\cup \{y,v\}$ are all minority sets of legal colorings and thus we are done by \oref{cover}.

We prove (iv) by contradiction. Assume $\cQ$ is a query set with $\delta_2(\cQ)\ge \frac{5n}{6}$ that does not determine a non-minority ball. Then by \oref{cover} there exists a set of answers for which the minority sets of all legal colorings cover the ball set $B$. Let $\cC$ be a minimal set of legal colorings for which $B=\cup_{c\in\cC}M_c$ holds. Note that $|\cC|\ge 3$ as $|M_c|<n/2$ for any legal coloring $c$. Let us consider three legal colorings $c_1,c_2,c_3 \in \cC$ and the corresponding minority sets $M_1, M_2, M_3$. By the minimality of $\cC$, there exist balls $b_i$ $i=1,2,3$ with $b_i\in M_{i}\setminus \cup_{c\in \cC, c\neq c_i}M_c$. We will use the following simple observation, that we include here without proof:

\begin{observation}\label{obbancs}
If $A_1,A_2,A_3\subseteq [n]$ with $|A_i|<n/2$, then for some $i\ne j$ $|A_i\cup A_j|< 5n/6$.
\end{observation}
%TUL TRIVI A BIZ, MINEK LEIRNI
%\begin{proof}%[Proof of Claim]
%We may assume that all $A_i$'s have size $\lceil n/2\rceil -1$. Towards a contradiction, suppose that the statement of the claim does not hold for some $A_1,A_2,A_3$. Then $|A_i\cap A_{i+1}|<n/6$. Thus $|[n]\setminus (A_{i+1}\cup A_{i+2})|\ge |A_i\setminus (A_{i+1}\cup A_{i+2})|\ge \lceil n/2\rceil -1 -2n/6\ge n/6-1$ holds - a contradiction
%\end{proof}

Applying Observation \ref{obbancs} to the sets $M_{1},M_{2},M_{3}$, we obtain without loss of generality that the set $T:=B\setminus (M_{1}\cup M_{2})$ has size strictly larger than $n/6$. We claim that $Q_t:=\{b_1,b_2,t\}\notin \cQ$ for any $t\in T$. Indeed, if $Q_t \in \cQ$, then the answer to $Q_t$ cannot be $b_1$ as $c_1$ is a legal coloring and $b_1 \in M_{1}$, while $b_2,t \notin M_{1}$. Similarly, $b_2$ cannot be the answer to $Q_t$ as $c_2$ is a legal coloring and $b_2 \in M_{2}$, while $b_1,t \notin M_{2}$. Finally, $t$ cannot be the answer to $Q_t$ as $t\in T$ and thus there is a legal coloring $c\in \cC$ such that $t\in M_c$ and by the choice of $b_1$ and $b_2$ we have $b_1,b_2 \notin M_c$. This shows that $d_\cQ(b_1,b_2)\le n-2-|T|\le \frac{5n}{6}-2$, a contradiction.

To prove (v), we need a construction.
First let us assume that $n$ is divisible by six.
Let $A_1,A_2,A_3,A_4,A_5,A_6$ be six pairwise disjoint subsets of $[n]$ with $|A_1|=|A_3|=|A_5|=n/6-1$ and $|A_2|=|A_4|=|A_6|=n/6+1$.
Then their union is $[n]$ and the sets $M_1:=A_1\cup A_2\cup A_3$, $M_2=A_3\cup A_4\cup A_5$, $M_3=A_5 \cup A_6\cup A_1$ all have size $n/2-1$.
Let us define
\[
\cQ=\binom{[n]}{3} \setminus \left( \left\{ \{x,y,z\}\mid x\in A_1,y\in A_3,z\in A_5\right\}
\cup \left\{ \{x,y,z\}\mid x\in A_2,y\in A_4,z\in A_6\right\}\right).
\]
We claim that there exists a set of answers to all the queries in $\cQ$ such that all $M_i$'s are possible minority sets and therefore by \oref{cover}, $\cQ$ does not determine a non-minority ball.
Indeed, if a query intersects an $A_i$ in at least two balls, then we can answer any of these balls.
Otherwise, the query contains one ball from both of two adjacent sets $A_i$ and $A_{i+1}$ (where $A_7=A_1$). By symmetry we can assume they are $a_1\in A_1$ and $a_2\in A_2$.
If the third ball is in $A_3$ or $A_4$, we can answer $a_2$, while if the third ball is in $A_5$ or $A_6$, we can answer $a_1$.
One can easily verify that the colorings $c_i$ with $c^{-1}(0)=M_i$ are all legal.
If $n=6r+i$ for some $1\leq i \leq 5$, then we use the above construction, such that for $j \leq i$ we add a ball to $A_{2j}$ if $j\leq 3$ and add a ball to $A_{2j-7}$ if $j\ge 4$.
\end{proof}

\vspace{3mm}

Now we prove our bounds on $N_3(n)$.

\vspace{2mm}

\noindent
\textbf{Theorem 4.}
\textit{(i) If $n$ is odd, then $N_3(n) \ge \frac{1}{2}\binom{n}{3}$.}

\vspace{1mm}

\textit{(ii)If $n$ is even, then $ N_3(n) \ge \frac {1}{8}	\binom{n-2}{3}$.}

\vspace{1mm}

\textit{(iii) $N_3(n)\le (\frac{5}{6}+o(1))\binom{n}{3}$.}

\vspace{1mm}

\begin{proof}[Proof \tref{degree}]

(i) To obtain the lower bound on $N_3(2k+1)$ we use a standard averaging argument. By \tref{lower} $(i)$ we know that if $\cQ$ determines a non-minority ball, then $\delta_2(\cQ)\ge k$ and therefore the number of pairs $(\{x,y\}, Q)$ with $x,y\in Q$, $Q\in \cQ$ is, on one hand, exactly $3|\cQ|$, and on the other hand is at least $k\binom{2k+1}{2}$. Rearranging, we obtain $|\cQ|\ge \frac{(2k+1)\cdot 2k \cdot k}{6} \ge \frac{1}{2}\binom{2k+1}{3}$.

(ii) We also use a standard averaging argument. Counting the number $M$ of ordered triples $(\{x,y\},\{u,v\},Q)$ with $Q\in\cQ$, $|Q\cap \{x,y\}|=|Q\cap \{u,v\}|=1$ and $x,u,y,v$ pairwise different, by \tref{lower} $(iii)$ we obtain that $M\ge \binom{n}{2}\binom{n-2}{2}(n/2-2)$ holds. On the other hand $M=6|\cQ|(n-3)(n-4)$. Taking $Q$ for which $|Q|$ is minimal, we get that $N_3(n)= |Q|\ge n(n-1)(n-4)/48\ge \frac{1}{8}\binom{n-2}{3}$.

(iii) By \tref{lower} $(iv)$, the upper bound on $N_3(n)$ follows from the existence of a query set $\cQ$ with $5n/6\le d_{\cQ}(x,y) \le 5n/6+o(n)$. The existence of such a query set can be seen easily using the random construction letting $Q\in \cQ$ with probability $5/6+n^{-1/3}$ for any triple $Q$ independently. (It is worth mentioning that there is a much stronger recent result on designs due to Keevash \cite{Kee}.)

\end{proof}

\section{Remarks on queries of larger size and other models}\label{sec:rem}
In this section we gather some information on the non-minority problem if the query size $q$ is greater than $3$, namely we prove
\tref{parka} and
\tref{parosok}. and also investigate some related models.

Throughout this section the number of colors is still always two.
Recall that in this model a query is a set of $q$ balls, the answer is either (the index of) a majority ball, or that there is no majority color in this subset, % (if the size of the subset is odd and there are only two colors, then this cannot happen),
and that we denote the number of queries needed to determine a non-minority ball by $A_q(n)$ in the adaptive model and by $N_q(n)$ in the non-adaptive model.
As many of the proofs and algorithms are very similar to the ones we used for the case $q=3$, we do not always provide all details.

By definition, the model is quite different for odd and even values of $q$.
%It turns out that the difficulty of finding a non-minority ball depends on the parity of $k$.
We first consider the case when $q$ is odd.

\vspace{3mm}

\noindent
\textbf{Theorem 5.}
We have $N_{2l+1}(n)=\Omega_l(n^3)$ and $A_{2l+1}(n)=O_l(n^2)$.

\begin{proof}
First note that we can simulate queries of size $2l+1$ by queries of size $3$.
That is, we can determine a possible answer for any query of size $2l+1$ by asking the necessary triples of the query.
This gives $N_{2l+1}(n)\ge N_3(n)/N_3(2l+1)=\Omega_l(n^3)$ using \tref{degree}.

For the adaptive case, we can modify Algorithm 2DB. Again we set aside the answer ball, add a new ball and then repeat the procedure. After $n-2l$ queries, among the $2l$ balls that have not been removed either we have $l$ red and $l$ blue balls, or all the other $n-2l$ balls have the same color. Assuming $n \ge 4l+1$, in both cases it is enough to find a non-minority ball among the other $n-2l$ balls, thus $A_{2l+1}(n) \le n-2l+A_{2l+1}(n-2l)$. Now the quadratic upper bound follows obviously if for $n\le 4l$ the problem can be solved at all.

If $n$ is even, it is enough to solve the problem for any $n-1$ of the balls; the resulting ball is in non-minority among the original set of balls.
Thus it is enough to solve the problem for $n$ odd.
For $n=2l+1$ the solution is obvious. Now we consider the case $n=2l+3$.

\begin{lemma}\label{midriff}
If $n=2l+3$ and all possible queries of size $2l+1$ have been asked, then we can find a non-minority ball.
\end{lemma}

\begin{proof} We will assume that the color classes contain $l+1$ or $l+2$ balls, and make sure that our output is a ball that appears as an answer to a query.
If the sizes of the color classes are more unbalanced, then no ball from the smaller class can appear as an answer to a query, hence our ball is still non-minority.

Let us assume that we cannot find a non-minority ball.
If a set $S$ of $l+1$ elements is the minority class, it has the following property: an answer to a query is in $S$ if and only if the query contains $S$.
Let $\mathcal{F}$ be the family of sets $S$, $|S|=l+1$, with the property that for a query $Q$ if the answer to $Q$ belongs to $S$, then $S \subset Q$ holds.
Sets in $\cF$ are the \textit{candidates} for being the minority set.
The assumption that we cannot find a non-minority ball means that every ball that has appeared as an answer for a query is contained in a member of $\mathcal{F}$.
For any ball $b$ we define $\cF_b=\{F \in \cF: b\in F\}$.

\begin{lemma}\label{upperdecker}
1. $\cF_a\cup \cF_b\subsetneq \mathcal{F}$ for every two balls $a,b$.

2. For every two balls $a,b$, there is a third ball $c$, such that $\cF_c=\mathcal{F}\setminus (\cF_a \cup \cF_b)$ holds.
\end{lemma}

\begin{proof} Let us consider two arbitrary balls $a$ and $b$, and the query $Q$ that avoids them. Let $c$ be the answer to $Q$. There must be a member $F$ of $\mathcal{F}$ containing $c$ (otherwise $c$ is a non-minority ball), and then $F \subset Q$ by the definition of $\mathcal{F}$. This proves 1. To prove 2, note that, by the definition of $\cF$ applied to $c$ and $Q$, every member of $\mathcal{F}$ that contains $c$ must be a subset of $Q$ showing $\cF_c\cap (\cF_a \cup \cF_b)=\emptyset$. On the other hand if a candidate set is a subset of $Q$ it must contain the answer to $Q$.
\end{proof}

To finish the proof of Lemma \ref{midriff}, let us consider two elements $a$ and $b$ of a set $F\in \mathcal{F}$.
By Lemma \ref{upperdecker} there is a ball $c$ with $\cF_c=\cF\setminus (\cF_a\cup \cF_b)$. Similarly, for $a$ and $c$ we can find another ball $d$ with $\cF_d=\cF\setminus (\cF_a\cup \cF_c)$. We know $d$ cannot be the same as $b$, since $F\in \cF_a,\cF_b$. Every member of $\mathcal{F}$ belongs to exactly one of $\cF_a$, $\cF_d$ and $\cF_c$, and also to at least one of $\cF_a$, $\cF_b$ and $\cF_c$. This implies $\cF_d \subseteq \cF_b$. Now for $b$ and $d$ we can find a third ball $e$ with $\cF_e=\cF\setminus (\cF_d\cup \cF_b)$, but then $\cF=\cF_b\cup \cF_e$ holds -- a contradiction.
\end{proof}

Returning to the proof of \tref{parka}, we can therefore simulate queries of size $2l+3$ by queries of size $2l+1$ using Lemma \ref{midriff}.
Now using the lemma for $n=2l+5$ and query size $2l+3$ we can also simulate queries of size $2l+5$ by queries of size $2l+3$, and thus by queries of size $2l+1$. A repeated application of this argument finishes the proof for any odd value of $n$.
\end{proof}

The algorithm can be improved by reusing some queries between different phases if we ask them in a ``binary tree'' branching way.
We do not go into details, as we could only achieve an $\tilde O(n^{\frac 32})$ bound, while we conjecture that an $O(n)$ bound holds also in this case.\\

%Let us now consider the case $k=2l$. For simplicity, if the answer to a query is that there is no majority, we say the answer is N.

\noindent
\textbf{Theorem 6.}
We have $N_{2l}(n) \leq n(n-2l)$ and $A_{2l}(n) \leq n-2l+1$.

\begin{proof}
We give a non-adaptive algorithm of length at most $n(n-2l)$. Let the balls be $b_1, b_2, \ldots , b_n$ and let us consider the queries \[\{ b_1, b_2, \ldots , b_{2l} \}, \{ b_2, b_3, \ldots , b_{2l+1} \}, \ldots , \{ b_n, b_1, \ldots , b_{2l-1} \}.\] If the answer to all these queries is a ball in majority, then all the answer balls have the same color (otherwise there should be a query for which the answer is that there is no majority in the set), and it is easy to see that any of these balls is a majority ball in the whole set. Thus we may assume that there is no majority in (say) $\{ b_1, b_2, \ldots , b_{2l} \}$. Now for $2l+1 \le i \le n$ let us consider the query $Q_i = \{ b_2, b_3, \ldots , b_{2l}, b_i \}$. If the answer to $Q_i$ is a ball, then $b_1$ and $b_i$ have different colors, while if the answer is that there is no majority in $Q_i$, then $b_1$ and $b_i$ have the same color. That is, using $n-2l$ further queries we can explore the whole distribution of the colors in $b_{2l+1}, b_{2l+2}, \ldots , b_n$ and now to show a non-minority ball is straightforward. Since we need the further $n-2l$ queries for every possible sets of the first round, the non-adaptive algorithm uses $n(n-2l)$ queries altogether, since we counted twice the \textit{interval} queries (e.g. $\{ b_2, b_3, \ldots , b_{2l+1} \}$ once at the beginning and then as $Q_{2l+1}$).

For the adaptive case let us start our algorithm with an arbitrary  query $S$ and suppose that the answer is a ball $s$. Then let the next query be $S'=S \setminus\{s\}\cup \{s'\}$, where $s' \not\in S$. If the answer to this query is a ball $b$, then $b$ and $s$ must have the same color. In this case we continue the process by taking a new ball $s'' \not\in S\cup \{s'\}$ and $S''=S'\setminus \{b\}\cup \{s''\}$. We do this (always deleting the answered ball from the current set and adding a new ball) until the answer is not a majority ball or there are no more balls to add. In the latter case we used $n-2l+1$ questions and a non-minority ball is obvious to show. Otherwise, when we stop in the current set $S_1$ of size $2l$ there is no majority ball and the set $S_2$ of all balls deleted so far (having size $i$ for some $0\le i \le n-2l$) is such that for some $a\in S_1$ (the ball added last to $S_1$) and for any $b\in S_2$ we know that $a$ and $b$ have different colors. Now it is easy to learn whether a ball $c\in B\setminus S_1 \setminus S_2$ has the same color as $a$: we ask the query $S_1\setminus \{a\} \cup \{c\}$, just like in the second phase of the non-adaptive algorithm we have just seen. That is, after using $n-2l+1$ questions altogether, we can easily show a non-minority ball again.
\end{proof}

\subsection*{Non-minority answers}
The huge difference between the odd and even cases motivates us to consider the model where the answer to a query is simply a non-minority ball.
If the query size $q$ is odd, then a non-minority answer is simply a majority answer, thus the above results remain true in this case, in particular, by asking all queries we can determine a non-minority ball. If $q=2$ or $q=4$ and $n$ is arbitrary or $q$ is even and $n$ is odd, then this is not possible by Lemma \ref{nonmin} below.
Finally, the remaining case, when $q>4$ and $n$ are both even, is handled for $n$ large enough by \tref{eveneven} (the proof of which is postponed to the Appendix), showing that in this case we can again find a non-minority ball.

%DOM 07.12: ES AZT EL LEHET DONTENI, HOGY VAN-E MAJORITY BALL? MAR EN SEM LATOM AT, HOGY MIRE MI VAN, LEHET, HOGY KENE EGY TABLAZATOT CSINALNI...

\begin{lemma}\label{nonmin} Assume that $q=2$, or $q=4$ and $n>2$ is arbitrary, or $q$ is even and $n$ is odd.
Even if we know a non-minority ball in all ${n \choose q}$ possible $q$-tuples, it is possible that we cannot show a non-minority ball.
\end{lemma}
\begin{proof}
If $q=2$, trivially no answer gives any information about the balls, and so it is impossible to determine a non-minority ball.

For $q=4$ and $n>2$ even, we partition thee balls into three groups of size $<n/2$, any query intersects at least one group in two elements, we always answer one of these elements. This way all colorings in which a group is monochromatic is valid, and so, by \oref{cover}, it is not possible to determine a non-minority ball.

For $n$ odd and $q$ even we partition tha balls into two groups of size $(n-1)/2$ and a group with one element, as in the previous case any query intersects at least one group in $q/2$ elements, we always answer one of these elements. This way again all colorings in which a group is monochromatic is valid, and so, by \oref{cover}, it is not possible to determine a non-minority ball.
\end{proof}

The proof of the next theorem is postponed to Appendix A.

\begin{theorem}\label{thm:eveneven} Let $n\ge q^3$ be even and $q>4$ be even.
If we know a non-minority ball in all ${n \choose q}$ possible $q$-tuples, we can find a non-minority ball.
\end{theorem}

Of course, it would be interesting to know whether \tref{eveneven} holds for all even $n$.
For this to hold, it would be sufficient to prove it for $n=q+2$, similarly as was done for \tref{parka} in Lemma \ref{midriff}.
Suppose that there is a collection $\mathcal{F}$ of minority sets, where the size of each set is at most $q/2$ and they cover each of the $n$ elements.
If a query $Q$ of size $q=n-2$ contains all elements but $x$ and $y$, then the answer to $Q$ can be $z$ if and only if for every $F\in \mathcal{F}$ with $z\in F$ we have $|F\setminus\{x,y\}|=q/2$.
This can be translated to the language of hypergraphs as follows.

A set $D$ of vertices is {\em dominating} in a hypergraph $\mathcal{H}$ if  for every vertex $v$ there exists an $H \in \cH$ with $v \in H$ and $H \cap D \neq \emptyset$.
The {\em dominating number} $\gamma(\mathcal{H})$ is the smallest number for which there is a dominating vertex set of size $\gamma(\mathcal{H})$.
In our case, the edges of $\mathcal{H}$ will be the sets of size {\em exactly} $q/2$ from $\mathcal{F}$, and the vertices of $\mathcal{H}$ the elements that are contained in such an edge.
Then a query $Q$ not containing $x$ and $y$ has an answer if and only if $\{x,y\}$ does not form a dominating set.
Therefore, $\mathcal{F}$ is a collection of minority sets if and only if for the corresponding hypergraph $\mathcal{H}$ we have $\gamma(\mathcal{H})\ge 3$.
This gives the following equivalent reformulation of our problem.

\begin{problem} Is it true that $\gamma(\mathcal{H})\le 2$ if $\mathcal{H}$ is $k$-uniform for some $k\ge n/2-1$?
\end{problem}

After the first version of this article appeared on arXiv, this problem was answered in the negative in \cite{BPTV}.

\subsection*{Arbitrary ratio model}

We finish this section with the introduction of a more general model. For some $0<\alpha \le 1/2$, we say that a ball $b \in B$ is an {\em $\alpha$-ball} of the set $B$ if there are at least $\alpha(|B|-1)$ other balls in the set that have the same color as $b$.
Similar density queries were considered in the group testing model in \cite{GKPW}.
Notice that a $\frac 12$-ball is just a majority ball while in a query of size $q$ for $\alpha=\frac 12\frac {q-2}{q-1}$ an $\alpha$-ball is exactly a non-minority ball.
In this new model, the queries are  subsets of size $q$ of the set of $n$ balls colored with two colors and the answer we obtain to a query $Q$ is an $\alpha$-ball of $Q$ (which always exists if we have only two colors).

Our goal is to find a ball having the same color as many other balls as possible.
Let $A$ be the largest integer, such that $\lceil \frac qA\rceil-1\ge \alpha(q-1)$ holds.
%The next statements show what is the best we can hope is {\em almost} a $\frac 1A$-ball.

%\lfloor \frac{k-1}{\lceil \alpha (k-1)\rceil}\rfloor-1$.
% for is practically a $\frac 1{\lfloor\frac 1\alpha\rfloor}$-ball.
% and our goal is to find an $\alpha$-ball of all $n$ balls.

\begin{prop}\label{prop:akkorisezazerosebb} For any $q$,$n$ and $0 < \alpha\le \frac12$, if $A$ divides $n$ or $q$, or $(n \mod A) < (q \mod A)$, then even asking all possible ${n \choose q}$ queries, it is possible that we cannot show a $\frac 1A$-ball of the full set.
%The same holds for other values for which $(n \mod A) < (k \mod A)$.
%EZT AZERT ELLENORIZZE LE MAS IS!
\end{prop}

Note that if $q$ is odd and $\alpha=\frac 12$, then $A=2$ and we can determine a $\frac 1A$-ball for odd values of $n$, showing that the above conditions are needed.
Also note that in the non-minority ball model, for $q=4$ we have $A=3$, but for $q>4$ we have $A=2$, which explains why we could find a non-minority ball only for $q>4$ and $n$ even  in \tref{eveneven}.

\begin{proof}
First let us consider the case when $n$ is any multiple of $A$.
In this case let us divide the balls into $A$ equal-sized sets $G_1,\ldots,G_A$.
The adversary reveals that the balls in the same group have the same color.
Any query contains at least $\lceil \frac qA\rceil$ balls from at least one of the groups, let this group be $G_i$. Any of these balls can be the answer to the query, as each have the same color as  $\lceil \frac qA\rceil-1\ge \alpha(q-1)$ other balls. %, where in the last inequality we have used that $\alpha (k-1)$ is an integer.
That is, receiving the answers it is possible that all balls have the same color and it is also possible that
in group $G_i$ there are only red balls and all the other balls are blue, therefore we cannot show a $\frac 1A$-ball.

In the above construction, if we add at most $q-A\cdot (\lceil \frac qA\rceil -1) -1$ further balls to our set, we still have the property that any query contains at least $\lceil \frac qA\rceil$ balls from one group. (Otherwise we would have at most $\lceil \frac qA\rceil -1$ balls from each group in the query and therefore at most a total of $A\cdot (\lceil \frac qA\rceil -1) + (q-A\cdot (\lceil \frac qA\rceil -1) -1)=q-1$ balls, a contradiction.)
%of whose color we do not reveal any information as they are never answers to any query.
It is easy to check that this settles all the remaining cases.
\end{proof}

%\pref{akkorisezazerosebb} almost completely deals with the non-minority model. %, consider a special case of the model for the case when $k$ is even. In this new model, the answer to a query is a non-minority ball. %This model is more natural in a sense because, similarly to the odd case, the relation of the query and the answer is the same as the relation of the input and output of the algorithm. In particular, if we have an algorithm that works on an input of size $n$ with queries of size $m$ and another one that works on an input of size $m$ with queries of size $k$, they naturally combine to make on algorithm on an input of size $n$ with queries of size $k$.
%On the other hand this model is less natural because it does not generalize the original majority problem. In fact for queries of size $2$ the answers do not give any meaningful information. And even for larger queries, they do not give enough information to find a non-minority ball.

The natural goal would be to find an $\alpha$-ball of the full set. As $\alpha \le \frac 1A$, \pref{akkorisezazerosebb} answers this question only for those values of $\alpha$ and $q$, where equality holds. The following theorem (the proof of which is postponed to the Appendix) shows that we can find an $\alpha$-ball for every other instances of $q$ and $\alpha$ if $n$ is large enough.
Moreover, $\frac 1A$ cannot be replaced by any smaller number.
To see this, we say that a ball $b\in B$ is a {\em $c$-almost} $\alpha$-ball of the set $B$ if there are at least $\alpha (|B|-1)-c$ other balls of the same color, where $c$ is some constant.

\begin{theorem}\label{thm:alfa}  If $n$ is large enough, after asking all possible ${n \choose q}$ queries, we can show a {\em $c_q$-almost} $\frac 1{A}$-ball of the full set,
%i.e., a ball such that there are at least $\frac nA-c_q$ other balls of the same color,
where $c_q$ is a constant depending only on $q$.
\end{theorem}

Finally, we  would like to modify the above model so that it makes sense for $1/2<\alpha \le 1$.
In this case it might happen that there is no $\alpha$-ball in a query. One possibility to deal with this is that the answer to a query $Q$ is an $\alpha$-ball of $Q$ if it exists, and any ball of $Q$ otherwise.
For example, if $q=4$ and $\alpha=2/3$, then if there are three red balls and one blue ball, then the answer is necessarily a red ball, but if there are two balls of each color, the answer is arbitrary.
Interestingly, a ball $b$ is a possible answer for a query $Q$ if and only if $b$ is a $(1-\alpha)$-ball in $Q$.
This way all the above results about $\alpha\le 1/2$ translate to this model, the best that we can hope for is an almost $\frac 1{A}$-ball of the whole set, where $A$ is the largest integer with $\lceil \frac qA\rceil-1\ge (1-\alpha)(q-1)$.

\subsection*{Open problems}

$\bullet$ Parts (i) and (ii) of \tref{lower} show that the minimum value of $\delta_2(\cQ)$ taken over all non-adaptive sets of queries determining a non-minority ball differs from $n/2$ by at most 1, provided $n$ is odd. We conjecture that this holds independently of the parity of $n$ (maybe with a larger constant instead of 1).

\vspace{2mm}

$\bullet$ Does \tref{eveneven} hold for every even $n$?

\vspace{2mm}

$\bullet$ Is it possible to remove the ``almost'' part from \tref{alfa}? (Assuming that the divisibility conditions required by \pref{akkorisezazerosebb} hold.)

\subsection*{Acknowledgement}
We would like to thank all participants of the Combinatorial Search Seminar at the Alfr\'ed R\'enyi Institute of Mathematics for fruitful discussions and the referees for providing useful comments which serve to improve the paper.

%\section*{Bibliography}

\bibliography{majomajorev}

\begin{thebibliography}{10}

\bibitem{A}
M.~Aigner.
\newblock Variants of the majority problem.
\newblock {\em Discrete Applied Mathematics}, 137(1):3--25, 2004.

\bibitem{ARS}
L.~Alonso, E.~M. Reingold, and R.~Schott.
\newblock Determining the majority.
\newblock {\em Information Processing Letters}, 47(5):253--255, 1993.

\bibitem{bfprt}
M.~Blum, R.~W. Floyd, V.~R. Pratt, R.~L. Rivest, and R.~E. Tarjan.
\newblock Time bounds for selection.
\newblock {\em Journal of Computer and System Sciences}, 7(4):448--461, 1973.

\bibitem{BPTV}
C.~Bujt{\'a}s, B.~Patk{\'o}s, Z.~Tuza, and M.~Vizer.
\newblock The minimum number of vertices in uniform hypergraphs with given
  domination number.
\newblock {\em Discrete Mathematics, to appear}, 2017.

\bibitem{DK}
G.~De~Marco and E.~Kranakis.
\newblock Searching for majority with $k$-tuple queries.
\newblock {\em Discrete Mathematics, Algorithms and Applications},
  7(02):1550009, 2015.

\bibitem{MKW}
G.~De~Marco, E.~Kranakis, and G.~Wiener.
\newblock Computing majority with triple queries.
\newblock {\em Theor. Comput. Sci.}, 461:17--26, 2012.

\bibitem{DZ}
D.~Dor and U.~Zwick.
\newblock Median selection requires (2+epsilon)$n$ comparisons.
\newblock {\em {SIAM} J. Discrete Math.}, 14(3):312--325, 2001.

\bibitem{DJKKS}
Z.~Dvo\v{r}\'{a}k, V.~Jel{\'{\i}}nek, D.~Kr\'{a}l', J.~Kyn\v{c}l, and M.~E.
  Saks.
\newblock Probabilistic strategies for the partition and plurality problems.
\newblock {\em Random Struct. Algorithms}, 30(1-2):63--77, 2007.

\bibitem{FS}
M.~J. Fischer and S.~L. Salzberg.
\newblock Finding a majority among $n$ votes.
\newblock {\em Journal of Algorithms}, 3:375--379, 1982.

\bibitem{GKPP}
D.~Gerbner, G.~O. Katona, D.~P{\'a}lv{\"o}lgyi, and B.~Patk{\'o}s.
\newblock Majority and plurality problems.
\newblock {\em Discrete Applied Mathematics}, 161(6):813--818, 2013.

\bibitem{GKPW}
D.~Gerbner, B.~Keszegh, D.~P{\'{a}}lv{\"{o}}lgyi, and G.~Wiener.
\newblock Density-based group testing.
\newblock In {\em Information Theory, Combinatorics, and Search Theory - In
  Memory of Rudolf Ahlswede}, pages 543--556, 2013.

\bibitem{HKM}
T.~P. Hayes, S.~Kutin, and D.~van Melkebeek.
\newblock The quantum black-box complexity of majority.
\newblock {\em Algorithmica}, 34(4):480--501, 2002.

\bibitem{JM}
R.~Johnson and G.~M\'esz\'aros.
\newblock Sorting using non-binary comparisons.
\newblock 2015.

\bibitem{Kee}
P.~Keevash.
\newblock The existence of designs.
\newblock {\em arXiv preprint arXiv:1401.3665}, 2014.

\bibitem{SW}
M.~E. Saks and M.~Werman.
\newblock On computing majority by comparisons.
\newblock {\em Combinatorica}, 11(4):383--387, 1991.

\bibitem{W}
G.~Wiener.
\newblock Search for a majority element.
\newblock {\em Journal of Statistical Planning and Inference}, 100(2):313--318,
  2002.

\end{thebibliography}
\bibliographystyle{abbrv}

\appendix

\section{Proof of \tref{eveneven}}

\noindent
\textbf{Theorem 16.}
\textit{Let $n\ge q^3$ be even and $q>4$ be even.
If we know a non-minority ball in all ${n \choose q}$ possible $q$-tuples, we can find a non-minority ball.}

\vspace{1mm}

\begin{proof}
We prove by contradiction.

Suppose we cannot find a non-minority ball.
Then by Observation \ref{obs:cover}, we know that the family $\mathcal{F}$ of minority sets of legal colorings covers the full set of balls.

\vspace{2mm}

First note that to get a contradiction, it is enough to show the existence of a set $Q$ with:

\vspace{1mm}

$(\bullet)$ $|Q|=q$ and for every ball $b \in Q$ there is a legal coloring $c(b)$ such that the minority set of that coloring $F_{c(b)}\in \mathcal{F}$ contains $b$ and $|F_{c(b)}\cap Q|<q/2$.

\vspace{1mm}

Indeed, if the answer to a query $Q$ is $x$, then the fact that $F_{c(x)}$
can be a color class together with the fact that $x \in F_{c(x)}$, implies that $F_{c(x)}$ is not in minority in $Q$, i.e., $|F_{c(x)}\cap Q| \ge q/2$. On the other hand $(\bullet)$ implies $|F_{c(x)}\cap Q|<q/2$, a contradiction.

%To find $Q$ with $(\bullet)$ is really enough: no element of $Q$ can be the answer for query $Q$, since if the answer is $x \in Q$, then $|F_{c(x)}\cap Q| \ge q/2$ by the definition of an answer and the fact that $x \in F_{c(x)}$, and that can not be by $(\bullet)$. So in the following we want to find a $Q$ satisfying $(\bullet)$.

\vspace{2mm}

Now let us consider a minimal subfamily $\mathcal{F}'$ of $\mathcal{F}$ which covers the ball set, i.e., no proper subfamily of $\mathcal{F}'$ covers all the balls.
The minimality of $\mathcal{F}'$ implies that for every $F\in \mathcal{F}'$ there is at least one ball contained in $F$ but in no other member of $\mathcal{F}'$.
For each $F\in \mathcal{F}'$ we call these balls, i.e., the balls that are not contained in other elements of $\mathcal{F}'$, the {\em private balls of $F$ for ${\cal F}'$} and denote them by $pb(F,{\cal F}')$.

If $|{\cal F}'|\ge q$, then we can choose $q$ private balls from $q$ different members of $\mathcal{F}'$. Their set $Q$ obviously satisfies $(\bullet)$, a contradiction. Thus we have $|\mathcal{F}'|<q$.

%With the help of private balls we prove that $|\mathcal{F}'|<q$.
%We prove it by contradiction. Suppose $|{\cal F}'|\ge q$, and consider a query set $R$ whose balls are private balls of $q$ different sets in ${\cal F}'$ (for ${\cal F}'$).
%If the answer to query $R$ is $x$, then consider the legal coloring $c(x)$ whose minority set $F_{c(x)}$ contains $x$ and $F_{c(x)} \in {\cal F}'$. Then by the assumptions on $q$ (here $2<q$ is enough) and the definition of the answer we know that $F_{c(x)}$ would contain other balls from $R$, so they would not be private balls of other elements of ${\cal F}'$ (for ${\cal F}'$), contradiction. Thus we have $|\mathcal{F}'|<q$.

\vspace{2mm}

%Now we start to constuct the desired $Q's$ satisfying $(\bullet)$.
If there is a set $F \in \mathcal{F'}$ with $|pb(F,\mathcal{F'})| \le q/2$, then we delete $F$ from $\mathcal{F'}$ and the private balls of $F$ from the set of balls, and we repeat this procedure for the remaining sets and balls until we can. Note that balls that are not private balls at the beginning can become private balls after deleting some sets and also note that every deleted ball is a private ball of some set at some point during the procedure.
Let $\mathcal{F}''$ be the family of all deleted sets and let $m:=|\mathcal{F}''|$, $\mathcal{G}:=\mathcal{F}' \setminus \mathcal{F}''$ and $X:=\cup_{F \in \mathcal{G}} F$. Note that $|X| \ge n-qm/2$ and $|\mathcal{G}| < q-m$ by $|\mathcal{F}'|<q$. Note also that $|F| \le n/2-1$, $q/2 < |pb(F, \mathcal{G})|$ for all $F \in \mathcal{G}$.

\vspace{2mm}

$\mathbf{Case \ 1:}$ First we consider $|\mathcal{G}|\le 1$. In this case set $X$ (whose cardinality is at least $n-qm/2> n-q^2/2$ (as $m<q$) by above), is covered by at most one set of size at most $n/2-1$. This is a contradiction, since the assumptions on $n$ and $q$ imply that $n/2-1 < n - q^2/2$.

\vspace{2mm}

$\mathbf{Case \ 2:}$ Now we consider the case $|\mathcal{G}|=2$ and let $\mathcal{G}:=\{A,B\}$.
The bound $|A\cup B|=|X|\ge n-mq/2 > n-q^2/2$ (using $m<q$) implies $$|A\setminus B|\ge n/2-q^2/2+1 \ \textrm{and} \ |B\setminus A|\ge n/2-q^2/2+1,$$ as members of $\mathcal{F}'$ have cardinality at most $n/2-1$ (and $n$ is even).
Every set $C\in \mathcal{F''}$ intersects either $A \setminus B$ or $B \setminus A$ in at most $n/4$ elements. Thus using that $n/2-q^2/2+1-n/4\ge q/2-1$ by the assumptions on $n$ and $q$, we have that for all $C\in \mathcal{F''}$

$$|(A\setminus B)\setminus C|\ge q/2-1 \ \textrm{or} \ |(B\setminus A)\setminus C|\ge q/2-1.$$

\vspace{1mm}

$\mathbf{Subcase \ 2.1:}$ $|\mathcal{F}''|=1$. Let $C$ be the only set in $\mathcal{F}''$ and without loss of generality we can assume, that $|(A\setminus B)\setminus C|\ge q/2-1$.
Since the size of the union of any two of the sets $A,B$ and $C$ is at most $n-2$ (here we use that $n$ is even), each of the sets has at least 2 private balls for $\mathcal{F}'$.
So put the $q/2-1$ balls of $(A\setminus B)\setminus C$, 2 private balls of $B$ and $C$ for $\mathcal{F}'$, and $q/2-3$ elements from the complement of $A$ into $Q$ (we note that this is a point, where we use the assumption $4<q$).
This $Q$ satisfies $(\bullet)$, since in this case $A \cup B \cup C$ contains all the balls, and $A$ is good minority set for those in $(A\setminus B)\setminus C$, $B$ and $C$ is a good minority set for those that are the private balls (for $\mathcal{F}'$) of $B$ and $C$ (respectively), and finally either $B$ or $C$ is a good minority set for those balls that are in the complement of $A$ and contains that ball.

\vspace{1mm}

$\mathbf{Subcase \ 2.2:}$ $|\mathcal{F}''| \ge 2 $ and there are $C,D\in \mathcal{F''}$ different with $|(A\setminus B)\setminus C|\ge q/2-1$ and $|(B\setminus A)\setminus D|\ge q/2-1$.
We can construct $Q$ by picking $q/2-1$ balls from both of $(A \setminus B)\setminus C$ and $(B \setminus A)\setminus D$, including at least one private ball of $A$ and $B$ for $\mathcal{F}'$, and in addition a private ball of both $C$ and $D$ for $\mathcal{F}'$.
This $Q$ satisfies $(\bullet)$, since $A$ ($B$) is a good minority set for the balls in $(A\setminus B)\setminus C$ ($(B \setminus A)\setminus D$, resp.) and $C$ and $D$ is a good minority sets for their private ball (respectively).

\vspace{1mm}

$\mathbf{Subcase \ 2.3:}$ $|\mathcal{F}''| \ge 2 $ and there are no such $C$ and $D$ as in the previous subcase. Then without loss of generality we can assume that for every $C\in \mathcal{F''}$ we have $|(A\setminus B)\setminus C|\ge q/2-1$ and $|(B\setminus A)\setminus C|\le q/2-2$. %let us pick an arbitrary $C\in \mathcal{F'}\setminus \{A,B\}$. We have, for example, $|(A\setminus B)\setminus C|\ge q/2-1$. Then for every $D\in \mathcal{F'}\setminus \{A,B,C\}$ we have $|(B\setminus A)\setminus D|\le q/2-2$ and also $|(A\setminus B)\setminus D|\ge q/2-1$, which in turn implies that $|(B\setminus A)\setminus C|\le q/2-2$.

Then for every $C\in  \mathcal{F''}$ we have $|(B\setminus A)\cap C|\ge n/2-q^2/2-q/2+3$ (using the bound on $|B \setminus A|$), and thus $|(A\setminus B)\cap C|\le q^2/2+q/2-4$ (using that $|C| \le n/2-1$).

\vspace{1mm}

Using our assumptions on $n$ and $q$ we have $(q^2/2+q/2-4)(q-3)+q/2-1 \le n/2-q^2/2+1$ (this is the place, where we really use the cubic lower bound on $n$) and $n/2-q^2/2+1 \le |A\setminus B|$, and we can conclude that the elements of $\mathcal{F''}$ together can cover all but at least $q/2-1$ balls from $A\setminus B$. This means that $A$ contains at least $q/2-1$ private balls for $\mathcal{F}'$. Then we can form $Q$ satisfying $(\bullet)$ by picking $q/2-1$ private balls of $A$ for $\mathcal{F}'$, a private ball of $B$ for $\mathcal{F}'$, a private ball of two different $C,D \in \mathcal{F''}$ for $\mathcal{F}'$, and $q/2-2$ further balls from the complement of $A$.
It is easy to see that this $Q$ satisfies $(\bullet)$ and we are done with $\mathbf{Case \ 2}$.

\vspace{2mm}

$\mathbf{Case \ 3:}$ Finally, if $|\mathcal{G}|\ge 3$,  let us consider 3 sets $A,B$ and $C$ with more than $q/2$ many private balls for $\mathcal{G}$. Then put $q/2-1$ private balls of $A$ and $B$ for $\mathcal{G}$ and 2 private balls of $C$ for $\mathcal{G}$ to form $Q$. This $Q$ easily satisfies $(\bullet)$.

\end{proof}

\section{Proof of \tref{alfa}}

\noindent
\textbf{Theorem 19.}
\textit{If $n$ is large enough, after asking all possible ${n \choose q}$ queries, we can show a {\em $c_q$-almost} $\frac 1{A}$-ball of the full set, where $c_q$ is a constant depending only on $q$.}

\vspace{1mm}

\begin{proof} We prove by contradiction, so suppose this is not the case.
Let $\cF$ be the family of sets that appear as color classes in colorings consistent with all the answers.
Then every ball must be contained in a member of $\cF$ that has less than $\frac nA-c_q$ elements.
As in the proof of \tref{eveneven}, let $\cF'$ be a subfamily of $\cF$ of minimum size that covers every ball, such that all members of $\cF'$ have less than $\frac nA-c_q$ elements.
Now for every $F\in \cF'$ there is a private ball, a ball contained in $F$ but no other member of $\cF'$.
As in \tref{eveneven}, $|\cF'|< q$. Indeed, otherwise we could find a query $Q$ of size $q$ containing only private balls and the answer $b$ for this query would be contained in $F \in \cF$.
But if $F$ would be a color class, then it would intersect $Q$ in one element concluding in a contradiction.

From now on we will consider colorings where a member of $\cF'$ is one of the color classes, and find a $c_q$-almost $\frac 1{A}$-ball of the full set, which is obviously a contradiction. If two balls are contained in exactly the same members of $\cF'$, their colors must be the same. Thus $\cF'$ determines at most $2^{q-1}$ groups such that all the balls in each group have the same color.
%We throw away the groups that have less than $\alpha(k-1)$ elements.
%If $|A|$ or less groups remain, any element of the largest group is an almost $\frac 1{A}$-ball. We consider only queries that intersect exactly $A+1$ groups, each in at most $\lceil \alpha(k-1)\rceil$ elements. For these queries the answer is an element of a group such that there is another group of the same color.

We partition each group into groups of size exactly $\lceil \alpha(q-1)\rceil$, throwing away less than $\lceil \alpha(q-1)\rceil$ balls of each original group. Let $\cG$ be the family of the resulting groups, then $|\cG| > \frac{ n-2^{q-1}(\lceil \alpha(q-1)\rceil)}{\lceil \alpha(q-1)\rceil}$.

Let us consider the queries that have a common element with exactly $A+1$ members of $\cG$, and contain no balls thrown away. One can easily see that $(A+1)(\lceil \alpha(q-1)\rceil) \ge q\ge A+1$, thus such queries exist. For these queries the answer is an element of a group, such that there is another group of the same color. Indeed, in a query there are at least $\alpha(q-1)+1$ elements with the color of the answered ball, not all of these fit into one group of size $\lceil \alpha(q-1)\rceil$. Considering the groups as elements of a new underlying set, this means we can ask queries of size $A+1$, and the answer to such a query is a $\frac 1{A}$-ball (or more precisely a $\frac 1{A}$-group) of the $A+1$ elements. If we can find a $\frac 1{A}$-group, it means that there are at least $\frac 1{A}(|\cG|-1)+1$ groups, thus $\frac nA-c_q$ balls, of the same color.

%Using that the groups are monochromatic, we can imagine that every group is a single Big ball, and represent these queries by size $A+1$ queries of Big balls, the answer to which is a Big ball that shares the color with at least one other Big ball.
%As the groups have different sizes, this would yield some weighted version of the problem, however, instead we can make several Medium balls equaling $\alpha(k-1)$ original balls of each group after throwing away at most $\alpha(k-1)$ balls of each group.
%So practically, we have shown that instead of an $\alpha$-ball, we can even get a $\frac 1A$-ball as the answer to our queries.
From here, the following algorithm run on the setting where elements correspond to the groups finishes the proof of the theorem.\\

{\bf Algorithm.}
Denote the current set of balls by $B$.
Ask a set of size $A+1$, and temporarily remove the answer from $B$.
Repeat this until only $A$ balls are left, whose set we denote by $K$.
If there are no red balls in $K$, then there can be no red balls in $B\setminus K$.
This means that any ball from $B\setminus K$ is a $\frac 1A$-ball of $B$.
If there is at least one ball of each color in $K$, then any $\frac 1A$-ball of $B\setminus K$ is also a $\frac 1A$-ball of $B$.
Therefore, in both cases, it is enough to solve the problem in $B\setminus K$.
Repeating this, using $O(n^2)$ queries we can achieve that we have $O(q)$ balls.
Notice that any of the remaining balls is a $c_q$-almost $\frac 1A$-ball of the full set.
\end{proof}

Note that the above proof is essentially the combination of the proof of \tref{eveneven} and the algorithm in the proof of \tref{parka}.
%We believe that similarly one could improve the bound to $\tilde O(n^{\frac 32})$, but again we conjecture that the truth is $O(n)$. DOM: ez hulyeseg volt, mert elejen feltettuk, hogy mindent megkerdeztunk.

%{\bf 0-1 model.} DOM: EZT ASSZEM EN IRTAM, DE MOST SEMMIT NEM ERTEK BELOLE, UGYHOGY KISZEDTEM.
%\begin{remark}
%If instead of the $\alpha$-ball model, we consider a model similar to \cite{JM} (discussed above \pref{median}) where our input is a set of (not necessarily different) $n$ numbers, a query is a set of $k$ numbers to which the answer is the $\lceil \alpha k\rceil$-th largest, then notice that any ball left after the ``first phase'' of the above algorithm is necessarily the $\alpha n\pm O(k)$-th largest out of the $n$ balls.
%\end{remark}

%The next statements show what is the best we can hope is {\em almost} a $\frac 1A$-ball.

\end{document}